\documentclass[11pt,a4paper]{amsart}

\newcommand{\C}{\mathbb{C}}
\newcommand{\R}{\mathbb{R}}
\newcommand{\N}{\mathbb{N}}

\usepackage{amssymb,amsmath,amsthm,enumerate}
\usepackage{graphicx}
\usepackage{color}

\usepackage{hyperref}
\usepackage{cite}

\usepackage{etoolbox}
\newtheorem{theorem}{Theorem}[section]
\newtheorem{lemma}[theorem]{Lemma}
\newtheorem{proposition}[theorem]{Proposition}
\newtheorem{corollary}[theorem]{Corollary}

\preto{\section}{}
\preto{\subsection}{}

\makeatletter
\@addtoreset{theorem}{chapter}
\@addtoreset{theorem}{section}
\@addtoreset{theorem}{subsection}
\makeatother

\usepackage{url}
\begin{document}

\title{Counting joints with multiplicities}
\author{Marina Iliopoulou}
\address{School of Mathematics and Maxwell Institute for Mathematical Sciences, University of Edinburgh, Edinburgh, EH9 3JZ, UK}
\email{\href{mailto:M.Iliopoulou@sms.ed.ac.uk}{M.Iliopoulou@sms.ed.ac.uk}}
\maketitle

\begin{abstract}

Let $\mathfrak{L}$ be a collection of $L$ lines in $\R^3$ and $J$ the set of joints formed by $\mathfrak{L}$, i.e. the set of points each of which lies in at least 3 non-coplanar lines of $\mathfrak{L}$. It is known that $|J| \lesssim L^{3/2}$ (first proved by Guth and Katz). For each joint $x \in J$, let the multiplicity $N(x)$ of $x$ be the number of triples of non-coplanar lines through $x$. We prove here that $\sum_{x \in J}N(x)^{1/2} \lesssim L^{3/2}$, while in the last section we extend this result to real algebraic curves in $\R^3$ of uniformly bounded degree, as well as to curves in $\R^3$ parametrised by real polynomials of uniformly bounded degree.

\end{abstract}
\section{Introduction}

A point $x \in \R^n$ is a joint for a collection $\mathfrak{L}$ of lines in $\R^n$ if there exist at least $n$ lines in $\mathfrak{L}$ passing through $x$, whose directions span $\R^n$. The problem of bounding the number of joints by a power of the number of the lines forming them first appeared in \cite{Chazelle_Edelsbrunner_Guibas_Pollack_Seidel_Sharir_Snoeyink_1992}, where it was proved that if $J$ is the set of joints formed by a collection of $L$ lines in $\R^3$, then $|J| =O(L^{7/4})$. Successive progress was made in improving the upper bound of $|J|$ in three dimensions, by Sharir, Sharir and Welzl, and Feldman and Sharir (see \cite{MR1280600},\cite{MR2047237},\cite{MR2121298}). Wolff  had already observed in \cite{MR1660476} that there exists a connection between the joints problem and the Kakeya problem, and, using this fact, Bennett, Carbery and Tao found an improved upper bound for $|J|$, with a particular assumption on the angles between the lines forming each joint (see \cite{MR2275834}). Eventually, Guth and Katz provided a sharp upper bound in \cite{Guth_Katz_2008}; they showed that, in $\R^3$, $|J|=O( L^{3/2})$. The proof was an adaptation of Dvir's algebraic argument  in \cite{MR2525780} for the solution of the finite field Kakeya problem, which involves working with the zero set of a polynomial. Dvir, Guth and Katz induced dramatic developments with this work, because they used for the first time the polynomial method to approach problems in incidence geometry. Further work was done by Elekes, Kaplan and Sharir in \cite{MR2763049}, and finally, a little later, Kaplan, Sharir and Shustin (in \cite{MR2728035}) and Quilodr\'{a}n (in \cite{MR2594983}) independently solved the joints problem in $n$ dimensions, using again algebraic techniques, simpler than in \cite{Guth_Katz_2008}.

In particular, Quilodr\'{a}n and Kaplan, Sharir and Shustin showed that, if $\mathfrak{L}$ is a collection of $L$ lines in $\R^n$, $n \geq 2$, and $J$ is the set of joints formed by $\mathfrak{L}$, then \begin{equation} |J| \leq c_n \cdot L^{\frac{n}{n-1}}, \label{eq:basic} \end{equation} where $c_n$ is a constant depending only on the dimension $n$. 

In this setting, we define the multiplicity $N(x)$ of a joint $x$ as the number of $n$-tuples of lines of $\mathfrak{L}$ through $x$, whose directions span $\R^n$; we mention here that we consider the $n$-tuples to be unordered, although considering them ordered would not cause any substantial change in what follows. 

From \eqref{eq:basic} we know that $ \sum_{x \in J}1 \leq c_n \cdot L^{\frac{n}{n-1}}$. A question by Anthony Carbery is if one can improve this to get
\begin{equation} \sum_{x \in J}N(x)^{\frac{1}{n-1}} \leq c_n' \cdot L^{\frac{n}{n-1}}, \label{eq:unsolved} \end{equation} where $c_n'$ is, again, a constant depending only on $n$. We clarify here that the choice of $\frac{1}{n-1}$ as the power of the multiplicities $N(x)$ on the left-hand side of \eqref{eq:unsolved} does not affect the truth of \eqref{eq:unsolved} when each joint has multiplicity $1$, while it is the largest power of $N(x)$ that one can hope for, since it is the largest power of $N(x)$ that makes \eqref{eq:unsolved} true when all the lines of $\mathfrak{L}$ are passing through the same point and each $n$ of them are linearly independent (in which case the point is a joint of multiplicity $\binom{L}{n} \sim L^n$). Also, \eqref{eq:unsolved} obviously holds when $n=2$; in that case, the left-hand side is smaller than the number of all the pairs of the $L$ lines, i.e. than $\binom{L}{2}\sim L^2$.

In fact, the above question can also be seen from a harmonic analytic point of view (again, see \cite{MR1660476}). Specifically, if $T_{\omega}$, for $\omega \in \Omega \subset S^{n-1}$, are tubes in $\R^n$ with length 1 and cross section an $(n-1)$-dimensional ball of radius $\delta$, such that their directions $\omega \in \Omega$ are $\delta$-separated, then the Kakeya maximal operator conjecture asks for a sharp upper bound of the quantity
\begin{displaymath}\int_{x \in \R^n} \Bigg(\sum_{\omega \in \Omega} \chi_{T_{\omega}}(x)\Bigg)^{\frac{n}{n-1}}{\rm d}x= \int_{x \in \R^n} \#\{ \text{tubes }T_{\omega}\text{ through }x\}^{\frac{n}{n-1}}{\rm d}x.\end{displaymath}
On the other hand, in the case where a collection $\mathfrak{L}$ of lines in $\R^n$ has the property that, whenever $n$ of the lines meet at a point, they form a joint there, then, for all $x \in J$, $N(x)\sim \#\{$lines of $\mathfrak{L}$ through $x\}^n$, and thus the left-hand side of \eqref{eq:unsolved} is
\begin{displaymath}\sim \sum_{x \in J}\#\{\text{lines of }\mathfrak{L}\text{ through }x\}^{\frac{n}{n-1}}.\end{displaymath}
Therefore, in both cases, the problem lies in bounding analogous quantities, and thus the problem of counting joints with multiplicities is a discrete analogue of the Kakeya maximal operator conjecture.

We will indeed show that \eqref{eq:unsolved} holds in $\R^3$: 

\begin{theorem}\label{1.1} Let $\mathfrak{L}$ be a collection of $L$ lines in $\R^3$, forming a set of joints $J$. Then,
\begin{displaymath} \sum_{x \in J}N(x)^{1/2} \leq c \cdot L^{3/2}, \end{displaymath}
where $c$ is a constant independent of $\mathfrak{L}$. \end{theorem}

The basic tool for the proof of Theorem \ref{1.1} will be the Guth-Katz polynomial method, developed by Guth and Katz in \cite{Guth_Katz_2010}.

More particularly, we first need to consider, for all $N \in \N$, the subset $J_N$ of $J$, defined as follows:

$J_N:=\{x \in J: N \leq N(x) <2N\}$.

In addition, we define, for all $N, k \in \N$, the following subset of $J_N$:

$J_N^{k}:=\{x \in J_N$: $x$ intersects at least $k$ and less than $2k$ lines of $\mathfrak{L}\}$. 

We will then apply the Guth-Katz polynomial method in the same way as in \cite{Guth_Katz_2010}, from which we will decuce that:

\begin{proposition}\label{1.2} \textit{If $\mathfrak{L}$ is a collection of $L$ lines in $\R^3$ and $N, k\in \N$, then} \begin{displaymath}|J^{k}_{N}| \cdot N^{1/2} \leq c \cdot \bigg(\frac{L^{3/2}}{k^{1/2}} + \frac{L}{k}\cdot N^{1/2}\bigg), \end{displaymath}
\textit{where $c$ is a constant independent of $\mathfrak{L}$, $N$ and $k$.}\end{proposition}

Theorem \ref{1.1} will then follow from Proposition \ref{1.2} (details are displayed at the end of section \ref{3}). Finally, in the somewhat independent section \ref{4} of the paper we generalise the statement of Theorem \ref{1.1} for joints formed by real algebraic curves in $\R^3$ of uniformly bounded degree, as well as curves in $\R^3$ parametrised by real polynomials of uniformly bounded degree (Theorem \ref{4.2.1} and Corollary \ref{parametrisedcurves}, respectively).

Note that we cannot apply for higher dimensions the proof of \eqref{eq:unsolved} for three dimensions that we are providing, as a crucial part of it is that the number of critical lines of an algebraic hypersurface in $\R^3$ is bounded (details in section \ref{section2.2}), a fact which we do not know if is true in higher dimensions.

We clarify here that, in whatever follows, any expression of the form $A \lesssim B$ or $A = O(B)$ means that there exists a non-negative constant $M$, depending only on the dimension, such that $A \leq M \cdot B$, while any expression of the form $A\lesssim_{b_1,...,b_m} B$ means that there exists a non-negative constant $M_{b_1,...,b_m}$, depending only on the dimension and $b_1$, ..., $b_m$, such that $A \lesssim M_{b_1,...,b_m}\cdot B$. In addition, any expression of the
form $A \gtrsim B$ or $A \gtrsim_{b_1,...,b_m} B$ means that $B \lesssim A$ or $B \lesssim_{b_1,...,b_m} A$, respectively. Finally, any expression of the form $A \sim B$ means that $A \lesssim B$ and $A \gtrsim B$, while expression of the form $A \sim_{b_1,...,b_m} B$ means that $A \lesssim_{b_1,...,b_m} B$ and $A \gtrsim_{b_1,...,b_m} B$. 

And now, before continuing with the proof, we present certain facts and tools which will prove useful to our goal.

\textbf{Acknowledgements.} I am deeply grateful to my supervisor, Anthony Carbery, for introducing me with the topic dealt with in this paper -- a question that was his. I also thank him for his help and support. I am also very grateful to my fellow PhD student Jes\'{u}s Martinez-Garcia, for carefully reviewing my thoughts on algebraic geometry, as well as for numerous long discussions on this field. Finally, I would like to thank S\'{a}ndor Kov\'{a}cs and Lucas Kaufmann who, by answering my questions on \url{www.mathoverflow.net}, directed me to the areas of real algebraic geometry appropriate for establishing the geometric background of the paper.

\section{Background}

\subsection{The Guth-Katz polynomial method}

As we have already mentioned, this technique will be the basic tool for our proof. The method can be applied in $\R^n$, for all $n \in \N$, and results in a decomposition of $\R^n$ by the zero set of a polynomial. All the details are fully explained in \cite{Guth_Katz_2010}, but we are presenting here the basic result and the theorems leading to it.

We start with a result of Stone and Tukey, known as the polynomial ham sandwich theorem (that is, in fact, a consequence of the Borsuk-Ulam theorem). In particular, we say that the zero set of a polynomial $p \in \R[x_1,...,x_n]$ bisects a set $U \subset \R^n$ of finite, positive volume, when the sets $U \cap \{p>0\}$ and $U \cap \{p<0\}$ have the same volume.

\begin{theorem}\emph{\textbf{(Stone, Tukey, \cite{MR0007036})}}\label{2.1.1} Let $d \in \N^*$, and $U_1,$ ..., $U_M$ be Lebesgue-measurable subsets of $\R^n$ of finite, positive volume, where $M=\binom{d+n}{n}-1$. Then, there exists a non-zero polynomial in $\R[x_1,...,x_n]$, of degree $\leq d$, whose zero set bisects each $U_i$.\end{theorem}

In analogy to this, if $S$ is a finite set of points in $\R^n$, we say that the zero set of a polynomial $p \in \R[x_1,...,x_n]$ bisects $S$ if the sets $S \cap \{p>0\}$ and $S \cap \{p<0\}$ each contain at most half of the points of $S$. Now, Guth and Katz, using the above theorem, proved the following.

\begin{corollary}\emph{\textbf{(Guth, Katz, \cite[Corollary 4.4]{Guth_Katz_2010})}}\label{2.1.2} Let $d \in \N^*$, and $S_1,$ ..., $S_M$ be disjoint, finite sets of points in $\R^n$, where $M=\binom{d+n}{n}-1$. Then, there exists a non-zero polynomial in $\R[x_1,...,x_n]$, of degree $\leq d$, whose zero set bisects each $S_i$. \end{corollary}

Another proof of this appears in \cite{KMS}, using \cite{MR1988723}.

The Guth-Katz polynomial method consists of successive applications of this last corollary. We now state the result of the application of the method, while its proof is the method itself.

\begin{theorem}\emph{\textbf{(Guth, Katz, \cite[Theorem 4.1]{Guth_Katz_2010})}}\label{2.1.3} Let $\mathfrak{G}$ be a finite set of $S$ points in $\R^n$, and $d>1$. Then, there exists a non-zero polynomial $p \in \R[x_1,...,x_n]$, of degree $\leq d$, whose zero set decomposes $\R^n$ in $\sim d^n$ cells, each of which contains $\lesssim S/d^n$ points of $\mathfrak{G}$.\end{theorem}

\begin{proof}We find polynomials $p_1,$ ..., $p_J \in \R[x_1,...,x_n]$, in the following way.

By Corollary \ref{2.1.2} applied to the finite set of points $\mathfrak{G}$, there exists a non-zero polynomial $p_1 \in \R[x_1,...,x_n]$, of degree $\lesssim 1^{1/n}$, whose zero set $Z_1$ bisects $\mathfrak{G}$. Thus, $\R^n\setminus Z_1$ consists of $2^1$ disjoint cells (the cell $\{p_1>0\}$ and the cell $\{p_1<0\}$), each of which contains $\lesssim S/2^1$ points of $\mathfrak{G}$.

By Corollary \ref{2.1.2} applied to the disjoint, finite sets of points $\mathfrak{G} \cap \{p_1>0\}$, $\mathfrak{G} \cap \{p_1<0\}$, there exists a non-zero polynomial $p_2 \in \R[x_1,...,x_n]$, of degree $\lesssim 2^{1/n}$, whose zero set $Z_2$ bisects $\mathfrak{G} \cap \{p_1>0\}$ and $\mathfrak{G} \cap \{p_1<0\}$. Thus, $\R^n\setminus (Z_1 \cup Z_2)$ consists of $2^2$ disjoint cells (the cells $\{p_1>0\}\cap\{p_2>0\}, \{p_1>0\}\cap\{p_2<0\}, \{p_1<0\}\cap\{p_2>0\}$ and $\{p_1<0\}\cap\{p_2<0\}$), each of which contains $\lesssim S/2^2$ points of $\mathfrak{G}$.

We continue in a similar way; by the end of the $j$-th step, we have produced non-zero polynomials $p_1,$ ..., $p_j$ in $\R[x_1,...,x_n]$, of degrees $\lesssim 2^{(1-1)/n},$ ..., $\lesssim 2^{(j-1)/n}$, respectively, such that  $\R^n \setminus (Z_1 \cup ...\cup Z_j)$ consists of $2^j$ disjoint cells, each of which contains $\lesssim S/2^j$ points of $\mathfrak{G}$.

We stop this procedure at the $J$-th step, where $J$ is such that the polynomial $p:=p_1 \cdots p_J$ has degree $\leq d$ and the number of cells in which $\R^n \setminus (Z_1 \cup...\cup Z_J)$ is decomposed is $\sim d^n$ (in other words, we stop when $2^{(1-1)/n}+2^{(2-1)/n}+...+2^{(J-1)/n} \lesssim d$ and $2^J \sim d^n$, for appropriate constants hiding behind the $\lesssim$ and $\sim$ symbols). The polynomial $p$ has the properties that we want (note that its zero set is the set $Z_1 \cup...\cup Z_J$).

\end{proof}

\subsection{More preliminaries}\label{section2.2}

The great advantage of applying the Guth-Katz polynomial method for decomposing $\R^3$ and, at the same time, a finite set of points $\mathfrak{G}$ in $\R^3$, does not only lie in the fact that it allows us to have a control over the number of points of $\mathfrak{G}$ in the interior of each cell; it lies in the fact that the surface that partitions $\R^3$ is the zero set of a polynomial. This immediately enriches our setting with extra structure, giving us control over many quantities, especially in three dimensions. In particular, the following holds.

\begin{theorem}\emph{\textbf{(Guth, Katz, \cite[Corollary 2.5]{Guth_Katz_2008})}}\label{2.2.1} \emph{(Corollary of B\'ezout's Theorem)} Let $p_1$, $p_2 \in \R[x,y,z]$. If $p_1$, $p_2$ do not have a common factor, then there exist at most $\deg p_1 \cdot \deg p_2$ lines simultaneously contained in the zero set of $p_1$ and the zero set of $p_2$. \end{theorem}

An application of this result enables us to bound the number of critical lines of an algebraic hypersurface in $\R^3$.

\textbf{Definition.} \textit{Let $p \in \R[x,y,z]$ be a non-zero polynomial of degree $\leq d$. Let $Z$ be the zero set of $p$.}

\textit{We denote by $p_{sf}$ the square-free polynomial we end up with, after eliminating all the squares appearing in the expression of $p$ as a product of irreducible polynomials in $\R[x,y,z]$.}

\textit{A }critical point\textit{ $x$ of $Z$ is a point of $Z$ for which $\nabla{p_{sf}}(x)=0$. Any other point of $Z$ is called a }regular point\textit{ of $Z$. A line contained in $Z$ is called a }critical line\textit{ if each point of the line is a critical point of $Z$.}

Note that, for any $p \in \R[x,y,z]$, the polynomials $p$ and $p_{sf}$ have the same zero set.

Moreover, if $x$ is a regular point of the zero set $Z$ of a polynomial $p \in \R[x,y,z]$, then, by the implicit function theorem, $Z$ is a manifold locally around $x$ and the tangent space to $Z$ at $x$ is well-defined; it is, in fact, the plane perpendicular to $\nabla{p_{sf}}(x)$ that passes through $x$.

An immediate corollary of Theorem \ref{2.2.1} is the following.

\begin{proposition}\emph{\textbf{(Guth, Katz, \cite[Proposition 3.1]{Guth_Katz_2008})}}\label{2.2.3} Let $p \in \R[x,y,z]$ be a non-zero polynomial of degree $\leq d$. Let $Z$ be the zero set of $p$. Then, $Z$ contains at most $d^2$ critical lines. \end{proposition}

\begin{proof}Since there are no squares in the expansion of $p_{sf}$ as a product of irreducible polynomials in $\R[x,y,z]$, it follows that $p_{sf}$ and $\nabla{p_{sf}}$ have no common factor. In other words, if $p_{sf}=p_1\cdots p_k$, where, for all $i \in \{1,...,k\}$, $p_i$ is an irreducible polynomial in $\R[x,y,z]$, then, for all $i\in\{1,...,k\}$, there exists some $g_i \in \Big\{\frac{\partial p_{sf}}{\partial x},\frac{\partial p_{sf}}{\partial y},\frac{\partial p_{sf}}{\partial z}\Big\}$, such that $p_i$ is not a factor of $g_i$. 

Now, let $l$ be a critical line of $Z$. It follows that $l$ lies in the zero set of $p_{sf}$, and therefore in the union of the zero sets of $p_1$, ..., $p_k \in \R[x,y,z]$; so, there exists $j \in \{1,...,k\}$, such that $l$ lies in the zero set of $p_j$. However, since $l$ is a critical line of $Z$, it is also contained in the zero set of $\nabla{p_{sf}}$, and thus in the zero set of $g_j$ as well. Therefore, $l$ lies simultaneously in the zero sets of the polynomials $p_j$ and $g_j\in \R[x,y,z]$.

It follows from the above that the number of critical lines of $Z$ is equal to at most $\sum_{i=1,...,k}L_i$, where, for all $i \in \{1,...,k\}$, $L_i$ is the number of lines simultaneously contained in the zero set of $p_i$ and the zero set of $g_i$ in $\R^3$. And since the polynomials $p_i$ and $g_i \in \R[x,y,z]$ do not have a common factor, Theorem \ref{2.2.1} implies that $L_i\leq \deg p_i \cdot \deg g_i \leq \deg p_i \cdot d$, for all $i \in \{1,...,k\}$. Thus, the number of critical lines of $Z$ is $\leq \sum_{i=1,...,k}\deg p_i \cdot d\leq \deg p \cdot d=d^2$.

\end{proof}

And finally:

\textbf{Definition.} \textit{Let $\mathcal{P}$ be a collection of points and $\mathfrak{L}$ a collection of lines in $\R^n$. We say that the pair $(p, l)$, where $p \in \mathcal{P}$ and $l\in \mathfrak{L}$, is an incidence between $\mathcal{P}$ and $\mathfrak{L}$, if $p \in l$. We denote by $I_{\mathcal{P}, \mathfrak{L}}$ the number of all the incidences between $\mathcal{P}$ and $\mathfrak{L}$.}

\begin{theorem}\emph{\textbf{(Szemer\'edi, Trotter, \cite{MR729791})}}\label{2.2.5} Let $\mathfrak{L}$ be a collection of $L$ lines in $\R^2$ and $\mathcal{P}$ a collection of $P$ points in $\R^2$. Then,
\begin{displaymath}I_{\mathcal{P}, \mathfrak{L}} \leq C \cdot ( |P|^{2/3}|L|^{2/3}+|P|+|L|), \end{displaymath}
where $C$ is a constant independent of $\mathfrak{L}$ and $\mathcal{P}$. \end{theorem}

This theorem first appeared in \cite{MR729791}; other, less complicated proofs have appeared since (see \cite{MR1464571} and \cite{KMS}; in fact, in \cite{KMS} the proof is a consequence of the Gut-Katz polynomial method).

An immediate consequence of the Szemer\'edi-Trotter theorem is the following.

\begin{corollary}\emph{\textbf{(Szemer\'edi, Trotter, \cite{MR729791})}}\label{2.2.6} Let $\mathfrak{L}$ be a collection of $L$ lines in $\R^2$ and $\mathcal{\mathfrak{G}}$ a collection of $S$ points in $\R^2$, such that each of them intersects at least $k$ lines of $\mathfrak{L}$, for $k \geq 2$. Then,
\begin{displaymath}S\leq c_0 \cdot (L^2 k^{-3}+Lk^{-1}), \end{displaymath}
where $c_0$ is a constant independent of $\mathfrak{L}$ and $\mathfrak{G}$. \end{corollary}

Note that Theorem \ref{2.2.5} and Corollary \ref{2.2.6} hold not only in $\R^2$, but in $\R^n$ as well, for all $n \in \N$, $n >2$, by projecting $\R^n$ on a generic plane.

We now continue with the proof of Theorem \ref{1.1} (via the proof of Proposition \ref{1.2}).

\section{Proof}\label{3}

We start by making certain observations.

\begin{lemma}\label{3.1} Let $x$ be a joint of multiplicity $N$ for a collection $\mathfrak{L}$ of lines in $\R^3$, such that $x$ lies in $\leq 2k$ of the lines. If, in addition, $x$ is a joint of multiplicity $\leq \frac{N}{2}$ for a subcollection $\mathfrak{L}'$ of the lines, or if it is not a joint at all for the subcollection $\mathfrak{L}'$, then there exist $\geq \frac{N}{1000 \cdot k^2}$ lines of $\mathfrak{L} \setminus \mathfrak{L}'$ passing through $x$. \end{lemma}

\begin{proof} Since the joint $x$ lies in $\leq 2k$ lines of $\mathfrak{L}$, its multiplicity $N$ is $\leq \binom{2k}{3}\leq 8k^3$.

Now, let $A$ be the number of lines of $\mathfrak{L} \setminus \mathfrak{L}'$ that are passing through $x$. We will show that $A \geq \frac{N}{1000 \cdot k^2}$. Indeed, suppose that $A \lneq \frac{N}{1000 \cdot k^2}$. Then,

$N=\Big|\Big\{\{l_1, l_2, l_3\}:$ $l_1$, $l_2$, $l_3$ $\in \mathfrak{L}$ are passing through $x$, and their directions span $\R^3 \Big\}\Big|=$\newline
$\Big|\Big\{\{l_1, l_2, l_3\}:$ $l_1$, $l_2$, $l_3$ $\in \mathfrak{L}'$ are passing through $x$, and their directions span $\R^3 \Big\}\Big|+ \Big|\Big\{\{l_1, l_2, l_3\}:$ $l_1$, $l_2$, $l_3$ $\in \mathfrak{L} \setminus \mathfrak{L}'$ are passing through $x$, and their directions span $\R^3 \Big\}\Big|+\Big|\Big\{\{l_1, l_2, l_3\}:$ the lines $l_1$, $l_2$ $\in \mathfrak{L}'$, $l_3$ $\in \mathfrak{L} \setminus \mathfrak{L}'$ are passing through $x$, and their directions span $\R^3 \Big\}\Big|+\Big|\Big\{\{l_1, l_2, l_3\}:$ the lines $l_1$, $l_2$ $\in \mathfrak{L} \setminus \mathfrak{L}'$, $l_3$ $\in \mathfrak{L}'$ are passing through $x$, and their directions span $\R^3 \Big\}\Big|\leq$\newline
$\leq \frac{N}{2}+\binom{A}{3} + \binom{A}{2} \cdot 2k + \binom{2k}{2} \cdot A\leq$\newline
$\leq \frac{N}{2} + A^3 + A^2 \cdot 2k + (2k)^2 \cdot A\leq$\newline
$\leq \frac{N}{2}+\big(\frac{N}{1000 \cdot k^2}\big)^3 + \big(\frac{N}{1000 \cdot k^2}\big)^2 \cdot 2k + (2k)^2 \cdot \frac{N}{1000 \cdot k^2}\leq$\newline
$\leq \frac{N}{2} + \frac{N}{8} + \frac{N}{8} + \frac{N}{8} \lneq N$ (what we use here is the fact that $N \leq 8k^3$). 

So, we are led to a contradiction, which means that $A \geq \frac{N}{1000 \cdot k^2}$.

\end{proof}

\begin{lemma}\label{3.2} Let $x$ be a joint of multiplicity $N$ for a collection $\mathfrak{L}$ of lines in $\R^3$, such that $x$ lies in $\leq 2k$ of the lines. Then, for every plane containing $x$, there exist $\geq \frac{N}{1000 \cdot k^2}$ lines of $\mathfrak{L}$ passing through $x$, which are not lying in the plane. \end{lemma}

\begin{proof} Let $\mathfrak{L}'$ be the set of lines in $\mathfrak{L}$ passing through $x$ and lying in some fixed plane. By Lemma \ref{3.1}, we know that there exist $\geq \frac{N}{1000 \cdot k^2}$ lines of $\mathfrak{L} \setminus \mathfrak{L}'$ passing through $x$, and, by the definition of $\mathfrak{L}'$, these lines do not lie in the plane. Therefore, there indeed exist at least $\frac{N}{1000 \cdot k^2}$ lines of $\mathfrak{L}$ passing through $x$ and lying outside the plane.

\end{proof}

We now continue with the proof of Proposition \ref{1.2}. We mention here that our argument will be based to a large extent on the proof of \cite[Theorem 4.7]{Guth_Katz_2010}. The following presentation, though, is self-contained, and it will be made clear whenever the techniques of \cite{Guth_Katz_2010} are being repeated.

\textbf{Proposition 1.2.} \textit{If $\mathfrak{L}$ is a collection of $L$ lines in $\R^3$ and $N$, $k\in \N$, then} \begin{displaymath} |J^{k}_{N}| \cdot N^{1/2} \leq c \cdot \bigg(\frac{L^{3/2}}{k^{1/2}} + \frac{L}{k}\cdot N^{1/2}\bigg),\end{displaymath}
\textit{where $c$ is a constant independent of $\mathfrak{L}$, $N$ and $k$.}

\begin{proof}

The proof will be done by induction on the number of lines of $\mathfrak{L}$. Indeed, let $L \in \N$. For $c$ a (non-negative) constant which will be specified later:

- For any collection of lines in $\R^3$ that consists of 1 line, \begin{displaymath} |J^k_{N}|\cdot N^{1/2} \leq c \cdot \bigg(\frac{1^{3/2}}{k^{1/2}} + \frac{1}{k}\cdot N^{1/2}\bigg), \; \forall \; N, \;k \in \N\end{displaymath} (this is obvious, in fact, for any $c \geq 0$, as in this case $J_{N}=\emptyset$, $\forall \; N \in \N$). 

- We assume that \begin{displaymath} |J_N^{k}| \cdot N^{1/2} \leq c \cdot \bigg(\frac{L'^{3/2}}{k^{1/2}} + \frac{L'}{k}\cdot N^{1/2}\bigg), \; \forall \; N,\; k \in \N,\end{displaymath} for any collection of $L'$ lines in $\R^3$, for any $L' \lneq L$.

- We will now prove that  \begin{equation} |J_N^{k}| \cdot N^{1/2} \leq c \cdot \bigg(\frac{L^{3/2}}{k^{1/2}} + \frac{L}{k}\cdot N^{1/2}\bigg),\; \forall \; N, \; k \in \N \label{eq:final}\end{equation}
for any collection of $L$ lines in $\R^3$.

We emphasise here that this last claim should and will be proved for the same constant $c$ as the one appearing in the first two steps of the induction process, provided that that constant is chosen to be sufficiently large.

Indeed, let $\mathfrak{L}$ be a collection of $L$ lines in $\R^3$, and fix $N$ and $k$ in $\N$. Also, for simplicity, let \begin{displaymath}\mathfrak{G}:=J_{N}^k \end{displaymath} and \begin{displaymath} S:=|J_{N}^k|\end{displaymath} for this collection of lines.

We now proceed in effectively the same way as in the proof of \cite[Theorem 4.7]{Guth_Katz_2010}.

Each point of $\mathfrak{G}$ has at least $k$ lines of $\mathfrak{L}$ passing through it, so, by the Szemer\'edi-Trotter theorem, $S\leq c_0 \cdot (L^2 k^{-3}+Lk^{-1})$, where $c_0$ is a constant independent of $\mathfrak{L}$, $N$ and $k$. Therefore:

If $\frac{S}{2}\leq c_0\cdot  Lk^{-1}$, then $S \cdot N^{1/2} \leq 2c_0 \cdot \frac{L}{k}\cdot N^{1/2}$ (where $2c_0$ is independent of $\mathfrak{L}$, $N$ and $k$). 

Otherwise, $\frac{S}{2}> c_0 \cdot Lk^{-1}$, so, by the Szemer\'edi-Trotter theorem, $\frac{S}{2}< c_0 \cdot L^2 k^{-3}$, which gives $S < 2c_0 \cdot  L^2 k^{-3}$. 

Therefore, $d:=AL^2S^{-1}k^{-3}$ is a quantity $> 1$ whenever $A\geq 2c_0$; we thus choose $A$ to be large enough for this to hold, and we will specify its value later. Now, applying the Guth-Katz polynomial method for this $d>1$ and the finite set of points $\mathfrak{G}$, we deduce that there exists a non-zero polynomial $p$, of degree $\leq d$, whose zero set $Z$ decomposes $\R^3$ in $\lesssim d^3$ cells, each of which contains $\lesssim Sd^{-3}$ points of $\mathfrak{G}$. We can assume that this polynomial is square-free, as eliminating the squares of $p$ does not inflict any change on its zero set. 

If there are $\geq 10^{-8}S$ points of $\mathfrak{G}$ in the union of the interiors of the cells, we are in the cellular case. Otherwise, we are in the algebraic case.

\textbf{Cellular case:} We follow the arguments in the proof of \cite[Lemma 4.8]{Guth_Katz_2010}, to fix $A$ and deduce that $S \cdot N^{1/2} \lesssim L^{3/2}k^{-1/2}$. More particularly:

There are $\gtrsim S$ points of $\mathfrak{G}$ in the union of the interiors of the cells. However, we also know that there exist $\lesssim d^3$ cells in total, each containing $\lesssim Sd^{-3}$ points of $\mathfrak{G}$. Therefore, there exist $\gtrsim d^3$ cells, with $\gtrsim Sd^{-3}$ points of $\mathfrak{G}$ in the interior of each. We call the cells with this property ``full cells". Now:

$\bullet$ If the interior of some full cell contains $\leq k$ points of $\mathfrak{G}$, then $Sd^{-3} \lesssim k$, so $S \lesssim L^{3/2}k^{-2}$,  and since $N \lesssim k^3$, we have that $S \cdot N^{1/2} \lesssim L^{3/2}k^{-1/2}$.

$\bullet$ If the interior of each full cell contains $\geq k$ points of $\mathfrak{G}$, then we will be led to a contradiction by choosing $A$ so large, that there will be too many intersections between the zero set $Z$ of $p$ and the lines of $\mathfrak{L}$ which do not lie in $Z$. Indeed:

Let $\mathfrak{L}_Z$ be the set of lines of $\mathfrak{L}$ which are lying in $Z$. Consider a full cell and let $S_{cell}$ be the number of points of $\mathfrak{G}$ in the interior of the cell, $\mathfrak{L}_{cell}$ the set of lines of $\mathfrak{L}$ that intersect the interior of the cell and $L_{cell}$ the number of these lines. Obviously, $\mathfrak{L}_{cell} \subset \mathfrak{L} \setminus \mathfrak{L}_Z$.

Now, each point of $\mathfrak{G}$ has at least $k$ lines of $\mathfrak{L}$ passing through it, therefore each point of $\mathfrak{G}$ lying in the interior of the cell has at least $k$ lines of $\mathfrak{L}_{cell}$ passing through it. Thus, since $S_{cell} \geq k$, we get that $ L_{cell} \geq k + (k-1) + (k-2) +...+1 \gtrsim k^2$, so 
\begin{displaymath} L^2_{cell}k^{-3} \gtrsim L_{cell} k^{-1}.\end{displaymath} 
But $k \geq 3$, so, by the Szemer\'edi-Trotter theorem,
\begin{displaymath} S_{cell} \lesssim L^2_{cell}k^{-3} + L_{cell} k^{-1}.\end{displaymath}
Therefore, $S_{cell} \lesssim L^2_{cell}k^{-3}$, so, since we are working in a full cell, $Sd^{-3} \lesssim L^2_{cell}k^{-3}$, and rearranging we see that 
\begin{displaymath} L_{cell} \gtrsim S^{1/2}d^{-3/2}k^{3/2}.\end{displaymath}

But each of the lines of $\mathfrak{L}_{cell}$ intersects the boundary of the cell at at least one point $x$, with the property that the induced topology from $\R^3$ to the intersection of the line with the closure of the cell contains an open neighbourhood of $x$; therefore, there are $ \gtrsim S^{1/2}d^{-3/2}k^{3/2}$ incidences of this form between $\mathfrak{L}_{cell}$ and the boundary of the cell (essentially, if a line $l$ intersects the interior of a cell, we can choose one arbitrary point of the intersection of the line with the interior of the cell and move along the line starting from that point until we reach the boundary of the cell for the first time; if $p$ is the point of the boundary that we reach through this procedure,
then the pair $(p, l)$ can be the incidence between the line and the boundary of the particular cell that we take into account; we do not count incidences between this line and the boundary of the particular cell, with the property that locally around the intersection point the line lies outside the cell). 

On the other hand, if $x$ is a point of $Z$ which belongs to a line intersecting the interior of a cell, such that the induced topology from $\R^3$ to the intersection of the line with the closure of the cell contains an open neighbourhood of $x$,  then there exists at most one other cell whose interior is also intersected by the line and whose boundary contains $x$, such that the induced topology from $\R^3$ to the intersection of the line with the closure of that cell contains an open neighbourhood of $x$. This, in fact, is the reason why we only considered a particular type of incidences. More particularly, we are not, in general, able to bound nicely the number of all the cells whose boundaries all contain a point $x$ and whose interiors are all intersected by a line $l$ containing $x$, as the line could enter the interior of each of the cells only far from the point $x$. We know, however, that there exist at most two cells whose boundaries contain $x$ and such that $l$ lies in both their interiors locally around $x$. And the union of the boundaries of all the cells is the zero set $Z$ of the polynomial $p$. 

So, if $I$ is the number of incidences between $Z$ and $\mathfrak{L} \setminus \mathfrak{L}_Z$, $I_{cell}$ is the number of incidences between $\mathfrak{L}_{cell}$ and the boundary of the cell, and $\mathcal{C}$ is the set of all the full cells (which, in our case, has cardinality $\gtrsim d^3$), then the above imply that
\begin{displaymath} I \gtrsim \sum_{cell \in \mathcal{C}} I_{cell} \gtrsim (S^{1/2}d^{-3/2}k^{3/2})\cdot d^3 \sim S^{1/2}d^{3/2}k^{3/2}.\end{displaymath}
On the other hand, if a line does not lie in the zero set $Z$ of $p$, then it intersects $Z$ at $\leq d$ points. Thus,
\begin{displaymath} I \leq L \cdot d. \end{displaymath}
This means that 
\begin{displaymath} S^{1/2}d^{3/2}k^{3/2} \lesssim L \cdot d, \end{displaymath}
which in turn gives $A\lesssim 1$. In other words, there exists some constant $C$, independent of $\mathfrak{L}$, $N$ and $k$, such that $A \leq C$. By fixing $A$ to be a number larger than $C$ (and of course $\geq 2c_0$, so that $d> 1$), we have a contradiction.

Therefore, in the cellular case there exists some constant $c_1$, independent of $\mathfrak{L}$, $N$ and $k$, such that  \begin{displaymath} S \cdot N^{1/2} \leq c_1 \cdot \frac{L^{3/2}}{k^{1/2}}.\end{displaymath}

\textbf{Algebraic case:} Let $\mathfrak{G}_1$ denote the set of points in $\mathfrak{G}$ which lie in $Z$. Here, $|\mathfrak{G}_1|>(1-10^{-8})S$. We now analyse the situation.

Since each point of $\mathfrak{G}_1$ intersects at least $k$ lines of $\mathfrak{L}$, \begin{displaymath} I_{\mathfrak{G}_1,\mathfrak{L}} > (1-10^{-8})Sk. \end{displaymath}
Now, let $\mathfrak{L}'$ be the set of lines in $\mathfrak{L}$ each of which contains $\geq \frac{1}{100}SkL^{-1}$ points of $\mathfrak{G}_1$. Each line of $\mathfrak{L} \setminus \mathfrak{L}'$ intersects fewer than $\frac{1}{100}SkL^{-1}$ points of $\mathfrak{G}_1$, thus
\begin{displaymath} I_{\mathfrak{G}_1,\mathfrak{L}\setminus \mathfrak{L}'} \leq |\mathfrak{L}\setminus \mathfrak{L}'| \cdot \frac{Sk}{100L} \leq \frac{1}{100}Sk.\end{displaymath}
Therefore, since $I_{\mathfrak{G}_1, \mathfrak{L}}=I_{\mathfrak{G}_1,\mathfrak{L} \setminus \mathfrak{L}'}+I_{\mathfrak{G}_1,\mathfrak{L}'}$, it follows that 
\begin{displaymath}I_{\mathfrak{G}_1,\mathfrak{L}'}> (1-10^{-8}-10^{-2})Sk. \end{displaymath}
Thus, there are $\gtrsim Sk$ incidences between $\mathfrak{G}_1$ and $\mathfrak{L}'$; this, combined with the fact that there exist $\leq S$ points of $\mathfrak{G}$ in total, each intersecting $\leq 2k$ lines of $\mathfrak{L}$, implies that there exist $\gtrsim S$ points of $\mathfrak{G}_1$, each intersecting $\gtrsim k$ lines of $\mathfrak{L}'$. 

Let us now take a moment to look for a practical meaning of this: $\sim S$ of our initial points each lie in $\sim k$ lines of $\mathfrak{L}'$, which is a subset of our initial set of lines $\mathfrak{L}$. Thus, if $\mathfrak{L}'$ is a strict subset of $\mathfrak{L}$, and if many of these points are joints for $\mathfrak{L}'$ with multiplicity $\sim N$, we can use our induction hypothesis for $\mathfrak{L}'$ and solve the problem if $|\mathfrak{L}'|$ is significantly smaller than $L$; however, before being able to tackle the problem in the rest of the cases, we need to extract more information. 

To that end, we will need to use appropriate, explicit constants now hiding behing the $\gtrsim$ symbols, which we therefore go ahead and find.

More particularly, let $\mathfrak{G}'$ be the set of points of $\mathfrak{G}_1$ each of which intersects $\geq \frac{1-10^{-8}-10^{-2}}{2}k$ lines of $\mathfrak{L}'$.

Then, \begin{displaymath}I_{\mathfrak{G}_1 \setminus \mathfrak{G}', \mathfrak{L}'} \leq |\mathfrak{G}_1 \setminus \mathfrak{G}'|\cdot  \frac{1-10^{-8}-10^{-2}}{2}k \leq  \frac{1-10^{-8}-10^{-2}}{2}Sk,\end{displaymath}

therefore, since $I_{\mathfrak{G}_1, \mathfrak{L}'}= I_{\mathfrak{G}_1 \setminus \mathfrak{G}', \mathfrak{L}'}+I_{\mathfrak{G}', \mathfrak{L}'}$, it follows that
\begin{displaymath}I_{\mathfrak{G}', \mathfrak{L}'} > \frac{1-10^{-8}-10^{-2}}{2}Sk. \end{displaymath}
And obviously,  $I_{\mathfrak{G}', \mathfrak{L}'} \leq |\mathfrak{G}'|\cdot 2k$. Therefore, $ \frac{1-10^{-8}-10^{-2}}{2}Sk< |\mathfrak{G}'|\cdot 2k$, and thus
\begin{displaymath} |\mathfrak{G}'|\geq \frac{1-10^{-8}-10^{-2}}{4}S; \end{displaymath}
in other words, there exist at least $\frac{1-10^{-8}-10^{-2}}{4}S$ points of $\mathfrak{G}_1$, each intersecting $\geq \frac{1-10^{-8}-10^{-2}}{2}k$ lines of $\mathfrak{L}'$. 

Now, each point of $\mathfrak{G}_1$ lies in $Z$, so it is either a regular or a critical point of $Z$. Let $\mathfrak{G}_{crit}$ be the set of points of $\mathfrak{G}_1$ that are critical points of $Z$, and  $\mathfrak{G}_{reg}$ the set of points of $\mathfrak{G}_1$ that are regular points of $Z$; then, $\mathfrak{G}_1=\mathfrak{G}_{crit} \sqcup \mathfrak{G}_{reg}$.

We are in one of the following two subcases.

\textbf{The regular subcase:} At least $\frac{10^{-8}}{4}S$ points of $\mathfrak{G}_1$ are regular points of $Z$ ($|\mathfrak{G}_{reg}| \geq \frac{10^{-8}}{4}S$).

What we actually need to continue is that $Z$ contains $\gtrsim S$ points of $\mathfrak{G}$ that are regular. Now, if $x \in \mathfrak{G}$ is a regular point of $Z$, there exists a plane through it, containing all those lines through the point that are lying in $Z$ (otherwise, the point would be a critical point of $Z$). And, since $x$ is a joint for $\mathfrak{L}$, of multiplicity $\geq N$, lying in $\leq 2k$ lines of $\mathfrak{L}$, by Lemma \ref{3.2} there exist $\gtrsim \frac{N}{k^2}$ lines of $\mathfrak{L}$ passing through $x$, which are not lying on the plane; this means that these lines are not lying in $Z$, and thus each of them contains $\leq d$ points of $\mathfrak{G}_1$. Therefore, the number of incidences between $\mathfrak{G}_1$ and $\mathfrak{L} \setminus \mathfrak{L}_{Z}$ is $\gtrsim S \cdot \frac{N}{k^2}$, but also $\leq |\mathfrak{L} \setminus \mathfrak{L}_{Z}| \cdot d \leq L \cdot d$. Thus, $S \cdot \frac{N}{k^2} \lesssim L\cdot d$, which implies that $S \cdot N^{1/2} \lesssim L^{3/2}k^{-1/2}$.

Therefore, \begin{displaymath} S \cdot N^{1/2} \leq c_2 \cdot \frac{L^{3/2}}{k^{1/2}}, \end{displaymath} 
for some constant $c_2$ independent of $\mathfrak{L}$, $N$ and $k$.

\textbf{The critical subcase:} Fewer than $\frac{10^{-8}}{4}S$ points of $\mathfrak{G}_1$ are regular points of $Z$ ($|\mathfrak{G}_{reg}| < \frac{10^{-8}}{4}S$). Now, either $|\mathfrak{L'}| \geq \frac{L}{100}$ or $|\mathfrak{L'}| < \frac{L}{100}$.

$\bullet$ \textbf{Suppose that} $\boldsymbol{|\mathfrak{L'}| \geq \frac{L}{100}}$ \textbf{.} 

(The basic arguments for the proof of this case appear in the proof of \cite[Proposition 4.7]{Guth_Katz_2010}.)

We notice that, if $\frac{1}{200}SkL^{-1} \leq d$, then we obtain $S \lesssim L^{3/2}k^{-2}$ by rearranging, so $S \cdot  N^{1/2} \lesssim L^{3/2}k^{-1/2}$ (as $N \lesssim k^3$). 

Therefore, we assume from now on that $\frac{1}{200}SkL^{-1} \geq d+1$. Then, each line of $\mathfrak{L}'$ contains at least $d+1\geq \deg p+1$ points of the zero set $Z$ of $p$ (as $\mathfrak{G}_1$ lies in $Z$), and thus each line of $\mathfrak{L}'$ lies in $Z$.

Now, we know that each line of $\mathfrak{L}'$ contains $\geq \frac{1}{100} SkL^{-1}$ points of $\mathfrak{G}_1$. Therefore, it either contains $\geq \frac{1}{200} SkL^{-1}$ points of $\mathfrak{G}_{crit}$ or $\geq \frac{1}{200} SkL^{-1}$ points of $\mathfrak{G}_{reg}$. But $|\mathfrak{L}'| \geq \frac{L}{100}$, so, if $\mathfrak{L}_{crit}$ is the set of lines in $\mathfrak{L}'$ each containing $\geq \frac{1}{200} SkL^{-1}$ points of $\mathfrak{G}_{crit}$ and $\mathfrak{L}_{reg}$ is the set of lines in $\mathfrak{L}'$ each containing $\geq \frac{1}{200} SkL^{-1}$  points of $\mathfrak{G}_{reg}$, then either $|\mathfrak{L}_{crit}| \geq \frac{L}{200}$ or $|\mathfrak{L}_{reg}|\geq \frac{L}{200}$.

Let us suppose that, in fact, $|\mathfrak{L}_{reg}|\geq \frac{L}{200}$. This means that the incidences between $\mathfrak{L}$ and the points in $\mathfrak{G}$ which are regular points of $Z$ number at least $\frac{L}{200} \cdot \frac{1}{200} SkL^{-1} = \frac{1}{4 \cdot 10^{4}}  Sk$. However, there exist fewer than $\frac{10^{-8}}{4}S$ points of $\mathfrak{G}$ which are regular points of $Z$, and therefore they contribute fewer than $\frac{10^{-8}}{4}S \cdot 2k=\frac{1}{2 \cdot 10^{8}}Sk \lneq \frac{1}{4 \cdot 10^{4}}  Sk$ incidences with $\mathfrak{L}$; so, we are led to a contradiction. Therefore, $|\mathfrak{L}_{reg}|\lneq \frac{L}{200}$.

Thus, $|\mathfrak{L}_{crit}| \geq \frac{L}{200}$. Now, each line of $\mathfrak{L}_{crit}$ contains $\geq \frac{1}{200} SkL^{-1} \gneq d$ critical points of $Z$, i.e. $\gneq d$ points where $p$ and $\nabla{p}$ are zero. However, both $p$ and $\nabla{p}$ have degrees $\leq d$. Therefore, if $l \in \mathfrak{L}_{crit}$, then $p$ and $\nabla{p}$ are zero across the whole line $l$, so each point of $l$ is a critical point of $Z$; in other words, $l$ is a critical line of $Z$. So, the number of critical lines of $Z$ is $\geq |\mathfrak{L}_{crit}| \geq \frac{L}{200}$. On the other hand, the number of critical lines of $Z$ is $\leq d^2$ (Proposition \ref{2.2.3}). Therefore,
\begin{displaymath} \frac{L}{200} \leq d^2, \end{displaymath} which gives $S \lesssim L^{3/2}k^{-3}$ after rearranging. Thus, $S N^{1/2} \lesssim L^{3/2}k^{-3/2}$ ($\lesssim  L^{3/2}k^{-1/2}$).

In other words, \begin{displaymath} S N^{1/2} \leq c_3 \cdot \frac{L^{3/2}}{k^{1/2}}, \end{displaymath} 
for some constant $c_3$ independent of $\mathfrak{L}$, $N$ and $k$.

$\bullet$ \textbf{Suppose that} $\boldsymbol{|\mathfrak{L'}| < \frac{L}{100}}$ \textbf{.}

Since fewer than $\frac{10^{-8}}{4}$ points of $\mathfrak{G}_1$ are regular points of $Z$, the same holds for the subset $\mathfrak{G}'$ of $\mathfrak{G}_1$. So, at least $\frac{1-2 \cdot 10^{-8}-10^{-2}}{4}S$ points of $\mathfrak{G}'$ are critical points of $Z$.

Now, each of the points of $\mathfrak{G}'$ is a joint for $\mathfrak{L}$ with multiplicity in the interval $[N,2N)$, so it is either a joint for $\mathfrak{L}'$ with multiplicity in the interval  $[N/2, 2N)$, or it is a joint for $\mathfrak{L}'$ with multiplicity $< N/2$, or it is not a joint for $\mathfrak{L}'$. Therefore, one of the following two subcases holds.

\textbf{1st subcase:} There exist at least $\frac{1-2 \cdot 10^{-8}-10^{-2}}{8}S$ critical points in $\mathfrak{G}'$ each of which is either a joint for $\mathfrak{L}'$ with multiplicity $<N/2$ or not a joint at all for $\mathfrak{L}'$. Let $\mathfrak{G}_2$ be the set of those points.

By Lemma \ref{3.1}, for each point $x \in \mathfrak{G}_2$ there exist $\geq \frac{N}{1000 \cdot k^2}$ lines of $\mathfrak{L} \setminus \mathfrak{L}'$ passing through $x$. 

Now, let $\mathfrak{L}_3$ be the set of lines in $\mathfrak{L} \setminus \mathfrak{L}'$, such that each of them contains $\leq d$ critical points of $Z$. Then, one of the following two holds.

\textbf{(1)} There exist $\geq \frac{1-2 \cdot 10^{-8}-10^{-2}}{16}S$ points of $\mathfrak{G}_2$ such that each of them has $\geq \frac{N}{2000 \cdot k^2}$ lines of $\mathfrak{L}_3$ passing through it. Then, 
\begin{displaymath}S \cdot \frac{N}{k^2} \lesssim I_{\mathfrak{G}_2, \mathfrak{L}_3} \leq |\mathfrak{L}_3| \cdot d \leq L \cdot d.  \end{displaymath}  
Rearranging, obtain $S \cdot N^{1/2} \lesssim L^{3/2}k^{-1/2}$.

\textbf{(2)} There exist $\geq \frac{1-2 \cdot 10^{-8}-10^{-2}}{16}S$ points of $\mathfrak{G}_2$ such that each of them has $\geq \frac{N}{2000 \cdot k^2}$ lines of $\big(\mathfrak{L} \setminus  \mathfrak{L}'\big) \setminus \mathfrak{L}_3$ passing through it. Each line of $\big(\mathfrak{L} \setminus  \mathfrak{L}'\big) \setminus \mathfrak{L}_3$ contains $<\frac{1}{100}SkL^{-1}$ points of $\mathfrak{G}_1$. Also, it contains $>d$ critical points of $Z$, so it is a critical line. But $Z$ contains $\leq d^2$ critical lines in total (by Proposition \ref{2.2.3}). Therefore,
\begin{displaymath}S \cdot \frac{N}{k^2} \lesssim I_{\mathfrak{G}_2, \big(\mathfrak{L} \setminus  \mathfrak{L}'\big) \setminus \mathfrak{L}_3} \leq d^2 \cdot \frac{1}{100}SkL^{-1}, \end{displaymath}
so $S\cdot  N^{1/2} \lesssim L^{3/2}k^{-1/2}$, by rearranging.

Thus, in this 1st subcase,
\begin{displaymath} S \cdot N^{1/2} \leq c_4 \cdot \frac{L^{3/2}}{k^{1/2}},  \end{displaymath} 
where $c_4$ is a constant independent of $\mathfrak{L}$, $N$ and $k$.

We are now able to define the constant $c$ appearing in our induction process; we let $c:=\max\{2c_0, c_1, c_2, c_3, c_4\}$. Note that, in any case that has been dealt with so far, 
\begin{displaymath} S \cdot N^{1/2} \leq c \cdot \bigg(\frac{L^{3/2}}{k^{1/2}}+\frac{L}{k}\cdot N^{1/2}\bigg),\end{displaymath}
and $c$ is, indeed, an explicit, non-negative constant, independent of $\mathfrak{L}$, $N$ and $k$.

\textbf{2nd subcase:} At least $\frac{1-2 \cdot 10^{-8}-10^{-2}}{8}S$ points of $\mathfrak{G}'$ are joints for $\mathfrak{L}'$ with multiplicity in the interval $[\frac{N}{2},2N)$. Then, either \textbf{(1)} or \textbf{(2)} hold.

\textbf{(1)} At least $\frac{1-2 \cdot 10^{-8}-10^{-2}}{16}S$ points of $\mathfrak{G}'$ are joints for $\mathfrak{L}'$ with multiplicity in the interval $[N, 2N)$. However, each point of $\mathfrak{G}'$ intersects at least $\frac{1-10^{-8}-10^{-2}}{2} k$ and fewer than $2k$ lines of $\mathfrak{L}'$. Therefore, either ($1i$), ($1ii$) or ($1iii$) hold.

($1i$) At least $\frac{1-2 \cdot 10^{-8}-10^{-2}}{48}S$ points of $\mathfrak{G}'$ are joints for $\mathfrak{L}'$ with multiplicity in the interval $[N,2N)$, such that each of them lies in at least $k$ and fewer than $2k$ lines of $\mathfrak{L'}$. Then, since $|\mathfrak{L}' |< \frac{L}{100}\lneq L$, it follows from our induction hypothesis that
\begin{displaymath} \frac{1-2 \cdot 10^{-8}-10^{-2}}{48}S \cdot N^{1/2} \leq c \cdot \bigg( \frac{|\mathfrak{L}'|^{3/2}}{k^{1/2}} + \frac{|\mathfrak{L}'|}{k} \cdot N^{1/2} \bigg) \leq \end{displaymath}
\begin{displaymath}\leq  c \cdot \bigg( \frac{(L/100)^{3/2}}{k^{1/2}} + \frac{(L/100)}{k} \cdot N^{1/2} \bigg).  \end{displaymath}
However, \begin{displaymath} \frac{48}{1-2 \cdot 10^{-8}-10^{-2}} \cdot \frac{1}{100^{3/2}} <1\end{displaymath} and
\begin{displaymath}\frac{48}{1-2 \cdot 10^{-8}-10^{-2}} \cdot \frac{1}{100}<1, \end{displaymath}

therefore \begin{displaymath} S \cdot N^{1/2} \leq  c \cdot \bigg( \frac{L^{3/2}}{k^{1/2}} + \frac{L}{k} \cdot N^{1/2} \bigg).\end{displaymath}

($1ii$) At least $\frac{1-2 \cdot 10^{-8}-10^{-2}}{48}S$ points of $\mathfrak{G}'$ are joints for $\mathfrak{L}'$ with multiplicity in the interval $[N,2N)$, such that each of them lies in at least $\frac{1-10^{-8}-10^{-2}}{2}k$ and fewer than $(1-10^{-8}-10^{-2})k$ lines of $\mathfrak{L'}$. So, since $|\mathfrak{L}' |< \frac{L}{100}\lneq L$, it follows from our induction hypothesis that
\begin{displaymath} \frac{1-2 \cdot 10^{-8}-10^{-2}}{48}S \cdot N^{1/2} \leq 
\end{displaymath}
\begin{displaymath}
\leq c \cdot \Bigg( \frac{|\mathfrak{L}'|^{3/2}}{\big(\frac{1-10^{-8}-10^{-2}}{2}k \big)^{1/2}} + \frac{|\mathfrak{L}'|}{\big(\frac{1-10^{-8}-10^{-2}}{2}k \big)} \cdot N)^{1/2} \Bigg) \leq \end{displaymath}
\begin{displaymath}\leq  c \cdot \Bigg( \frac{(L/100)^{3/2}}{\big(\frac{1-10^{-8}-10^{-2}}{2}k \big)^{1/2}} + \frac{(L/100)}{\big(\frac{1-10^{-8}-10^{-2}}{2}k \big)} \cdot N^{1/2} \Bigg) . \end{displaymath}
However,
\begin{displaymath} \frac{48}{1-2 \cdot 10^{-8}-10^{-2}} \cdot \frac{1}{100^{3/2}} \cdot \frac{2^{1/2}}{(1-10^{-8}-10^{-2})^{1/2}} <1 \end{displaymath} and
\begin{displaymath}\frac{48}{1-2 \cdot 10^{-8}-10^{-2}} \cdot \frac{1}{100}\cdot \frac{2}{1-10^{-8}-10^{-2}}<1, \end{displaymath}
therefore 
\begin{displaymath} S \cdot  N^{1/2} \leq  c \cdot \bigg( \frac{L^{3/2}}{k^{1/2}} + \frac{L}{k} \cdot N^{1/2} \bigg).\end{displaymath}

($1iii$) At least $\frac{1-2 \cdot 10^{-8}-10^{-2}}{48}S$ points of $\mathfrak{G}'$ are joints for $\mathfrak{L}'$ with multiplicity in the interval $[N,2N)$, such that each of them lies in between $(1-10^{-8}-10^{-2})k$ and $2 \cdot (1-10^{-8}-10^{-2})k$ lines of $\mathfrak{L'}$. So, since $|\mathfrak{L}' |< \frac{L}{100}\lneq L$, it follows from our induction hypothesis that
\begin{displaymath} \frac{1-2 \cdot 10^{-8}-10^{-2}}{48}S \cdot N^{1/2} \leq c \cdot \Bigg( \frac{|\mathfrak{L}'|^{3/2}}{\big((1-10^{-8}-10^{-2})k)^{1/2}} +\end{displaymath}
\begin{displaymath}+ \frac{|\mathfrak{L}'|}{(1-10^{-8}-10^{-2})k} \cdot N^{1/2} \Bigg) \leq \end{displaymath}
\begin{displaymath}\leq  c \cdot \Bigg( \frac{(L/100)^{3/2}}{\big((1-10^{-8}-10^{-2})k \big)^{1/2}} + \frac{(L/100)}{(1-10^{-8}-10^{-2})k} \cdot N^{1/2} \Bigg) . \end{displaymath}
However,
\begin{displaymath} \frac{48}{1-2 \cdot 10^{-8}-10^{-2}} \cdot \frac{1}{100^{3/2}} \cdot \frac{1}{(1-10^{-8}-10^{-2})^{1/2}} <1 \end{displaymath} and
\begin{displaymath}\frac{48}{1-2 \cdot 10^{-8}-10^{-2}} \cdot \frac{1}{100}\cdot \frac{1}{1-10^{-8}-10^{-2}}<1, \end{displaymath}
therefore 
\begin{displaymath} S \cdot N^{1/2} \leq  c \cdot \bigg( \frac{L^{3/2}}{k^{1/2}} + \frac{L}{k} \cdot N^{1/2} \bigg).\end{displaymath}

\textbf{(2)} At least $\frac{1-2 \cdot 10^{-8}-10^{-2}}{16}S$ points of $\mathfrak{G}'$ are joints for $\mathfrak{L}'$ with multiplicity in the interval $[\frac{N}{2},N)$. However, each point of $\mathfrak{G}'$ intersects at least $\frac{1-10^{-8}-10^{-2}}{2}\cdot k$ and fewer than $2k$ lines of $\mathfrak{L}'$. Therefore, either ($2i$), ($2ii$) or ($2iii$) hold.

($2i$)  At least $\frac{1-2 \cdot 10^{-8}-10^{-2}}{48}S$ points of $\mathfrak{G}'$ are joints for $\mathfrak{L}'$ with multiplicity in the interval $[\frac{N}{2},N)$, such that each of them lies in at least $k$ and fewer than $2k$ lines of $\mathfrak{L'}$. Then, since $|\mathfrak{L}' |< \frac{L}{100}\lneq L$, it follows from our induction hypothesis that
\begin{displaymath} \frac{1-2 \cdot 10^{-8}-10^{-2}}{48}S \cdot \bigg(\frac{N}{2}\bigg)^{1/2} \leq c \cdot \Bigg( \frac{|\mathfrak{L}'|^{3/2}}{k^{1/2}} + \frac{|\mathfrak{L}'|}{k} \cdot \bigg(\frac{N}{2}\bigg)^{1/2} \Bigg) \leq \end{displaymath}
\begin{displaymath}\leq  c \cdot \Bigg( \frac{(L/100)^{3/2}}{k^{1/2}} + \frac{(L/100)}{k} \cdot \bigg(\frac{N}{2}\bigg)^{1/2} \Bigg).  \end{displaymath}
However, \begin{displaymath} \frac{48}{1-2 \cdot 10^{-8}-10^{-2}} \cdot \frac{1}{100^{3/2}} \cdot 2^{1/2} <1\end{displaymath} and
\begin{displaymath}\frac{48}{1-2 \cdot 10^{-8}-10^{-2}} \cdot \frac{1}{100}<1, \end{displaymath}
therefore 
\begin{displaymath} S \cdot N^{1/2} \leq  c \cdot \bigg( \frac{L^{3/2}}{k^{1/2}} + \frac{L}{k} \cdot N^{1/2} \bigg).\end{displaymath}

($2ii$) At least $\frac{1-2 \cdot 10^{-8}-10^{-2}}{48}S$ points of $\mathfrak{G}'$ are joints for $\mathfrak{L}'$ with multiplicity in the interval $[\frac{N}{2},N)$, such that each of them lies in at least $\frac{1-10^{-8}-10^{-2}}{2}\cdot k$ and fewer than $(1-10^{-8}-10^{-2})k$ lines of $\mathfrak{L'}$. So, since $|\mathfrak{L}' |< \frac{L}{100}\lneq L$, it follows from our induction hypothesis that
\begin{displaymath} \frac{1-2 \cdot 10^{-8}-10^{-2}}{48}S \cdot \bigg(\frac{N}{2}\bigg)^{1/2} \leq 
\end{displaymath}
\begin{displaymath} \leq c \cdot \Bigg( \frac{|\mathfrak{L}'|^{3/2}}{\big(\frac{1-10^{-8}-10^{-2}}{2}k \big)^{1/2}} + \frac{|\mathfrak{L}'|}{\big(\frac{1-10^{-8}-10^{-2}}{2}k \big)} \cdot \bigg(\frac{N}{2}\bigg)^{1/2} \Bigg) \leq \end{displaymath}
\begin{displaymath}\leq  c \cdot \Bigg( \frac{(L/100)^{3/2}}{\big(\frac{1-10^{-8}-10^{-2}}{2}k \big)^{1/2}} + \frac{(L/100)}{\big(\frac{1-10^{-8}-10^{-2}}{2}k \big)} \cdot \bigg(\frac{N}{2}\bigg)^{1/2} \Bigg) . \end{displaymath}
However,
\begin{displaymath} \frac{48}{1-2 \cdot 10^{-8}-10^{-2}} \cdot \frac{1}{100^{3/2}} \cdot \frac{2^{1/2}}{(1-10^{-8}-10^{-2})^{1/2}} \cdot 2^{1/2} <1 \end{displaymath} and
\begin{displaymath}\frac{48}{1-2 \cdot 10^{-8}-10^{-2}} \cdot \frac{1}{100}\cdot \frac{2}{1-10^{-8}-10^{-2}}<1, \end{displaymath}
therefore 
\begin{displaymath} S \cdot N^{1/2} \leq  c \cdot \bigg( \frac{L^{3/2}}{k^{1/2}} + \frac{L}{k} \cdot N^{1/2} \bigg).\end{displaymath}

($2iii$) At least $\frac{1-2 \cdot 10^{-8}-10^{-2}}{48}S$ points of $\mathfrak{G}'$ are joints for $\mathfrak{L}'$ with multiplicity in the interval $[\frac{N}{2},N)$, such that each of them lies in at least $(1-10^{-8}-10^{-2})k$ and fewer than $2 \cdot (1-10^{-8}-10^{-2})k$ lines of $\mathfrak{L'}$. So, since $|\mathfrak{L}' |< \frac{L}{100}\lneq L$, it follows from our induction hypothesis that
\begin{displaymath} \frac{1-2 \cdot 10^{-8}-10^{-2}}{48}S \cdot \bigg(\frac{N}{2}\bigg)^{1/2} \leq c \cdot \Bigg( \frac{|\mathfrak{L}'|^{3/2}}{\big((1-10^{-8}-10^{-2})k)^{1/2}} + \end{displaymath}
\begin{displaymath}+\frac{|\mathfrak{L}'|}{(1-10^{-8}-10^{-2})k} \cdot \bigg(\frac{N}{2}\bigg)^{1/2} \Bigg) \leq \end{displaymath}
\begin{displaymath}\leq  c \cdot \Bigg( \frac{(L/100)^{3/2}}{\big((1-10^{-8}-10^{-2})k \big)^{1/2}} + \frac{(L/100)}{(1-10^{-8}-10^{-2})k} \cdot \bigg(\frac{N}{2}\bigg)^{1/2} \Bigg) . \end{displaymath}
However,
\begin{displaymath} \frac{48}{1-2 \cdot 10^{-8}-10^{-2}} \cdot \frac{1}{100^{3/2}} \cdot \frac{2^{1/2}}{(1-10^{-8}-10^{-2})^{1/2}} <1 \end{displaymath} and
\begin{displaymath}\frac{48}{1-2 \cdot 10^{-8}-10^{-2}} \cdot \frac{1}{100}\cdot \frac{1}{1-10^{-8}-10^{-2}}<1, \end{displaymath}
therefore 
\begin{displaymath} S \cdot N^{1/2} \leq  c \cdot \bigg( \frac{L^{3/2}}{k^{1/2}} + \frac{L}{k} \cdot N^{1/2} \bigg).\end{displaymath}

We have by now exhausted all the possible cases; in each one,

\begin{displaymath} S \cdot N^{1/2} \leq  c \cdot \bigg( \frac{L^{3/2}}{k^{1/2}} + \frac{L}{k} \cdot N^{1/2} \bigg), \end{displaymath}

where $c$ is, by its definition, a constant independent of $\mathfrak{L}$, $N$ and $k$.

Therefore, as $N$ and $k$ were arbitrary, \eqref{eq:final} holds for this collection $\mathfrak{L}$ of lines in $\R^3$. And since $\mathfrak{L}$ was an arbitrary collection of $L$ lines, \eqref{eq:final} holds for any collection $\mathfrak{L}$ of $L$ lines in $\R^3$.

Consequently, the proposition is proved.

\end{proof}

Now, Theorem \ref{1.1} easily follows.

\textbf{Theorem 1.1.} \textit{Let $\mathfrak{L}$ be a collection of $L$ lines in $\R^3$, forming a set of joints $J$. Then,} 
\begin{displaymath} \sum_{x \in J}N(x)^{1/2} \leq c \cdot L^{3/2},\end{displaymath} 
\textit{where $c$ is a constant independent of $\mathfrak{L}$.}

\begin{proof}

The multiplicity of each joint in $J$ can be at most $ \binom{L}{3}\leq L^3$. Therefore,
\begin{displaymath} \sum_{x \in J}N(x)^{{1/2}}\leq 2 \cdot \sum_{\{\lambda \in \N:\; 2^{\lambda} \leq L^3\}}|J_{2^{\lambda}}| \cdot (2^{\lambda})^{1/2}.\end{displaymath}
However, if $x$ is a joint for $\mathfrak{L}$ with multiplicity $N$, such that fewer than $2k$ lines of $\mathfrak{L}$ are passing through $x$, then $N \leq \binom{2k}{3} \leq (2k)^3$, and thus $k \geq \frac{1}{2}N^{1/3}$. Therefore, for all $\lambda \in \N$ such that $2^{\lambda} \leq L^3$, 
\begin{displaymath} |J_{2^{\lambda}}|=\sum_{\big\{\mu \in \N: \; 2^{\mu}\geq \frac{1}{2}(2^{\lambda})^{1/3} \big\}}|J^{2^\mu}_{2^{\lambda}}|,\end{displaymath} thus
\begin{displaymath} |J_{2^{\lambda}}|\cdot (2^{\lambda})^{1/2} =\sum_{\big\{\mu \in \N: \; 2^{\mu}\geq \frac{1}{2}(2^{\lambda})^{1/3}\big\}}|J^{2^\mu}_{2^{\lambda}}|\cdot (2^{\lambda})^{1/2},\end{displaymath}
a quantity which, by Proposition \ref{1.2}, is 
\begin{displaymath} \leq \sum_{\big\{\mu \in \N:\; 2^{\mu}\geq \frac{1}{2}(2^{\lambda})^{1/3}\big\}} c \cdot \Bigg(\frac{L^{3/2}}{(2^{\mu}) ^{1/2}} + \frac{L}{2^{\mu}}\cdot (2^{\lambda})^{1/2}\Bigg) \leq \end{displaymath}
\begin{displaymath} \leq c' \cdot \Bigg(\frac{L^{3/2}}{\big((2^{\lambda})^{1/3}\big)^{1/2}} + \frac{L}{(2^{\lambda})^{1/3}}\cdot (2^{\lambda})^{1/2}\Bigg), \end{displaymath} 
where $c'$ is a constant independent of $\mathfrak{L}$, $k$ and $\lambda$.

Therefore, 
\begin{displaymath} \sum_{x \in J} N(x)^{1/2}\leq 2c'\cdot \sum_{\{\lambda \in \N:\; 2^{\lambda}\leq L^3\}} \Bigg(\frac{L^{3/2}}{(2^{\lambda})^{1/6}} + L\cdot (2^{\lambda})^{1/6}\Bigg) \leq 
\end{displaymath}
\begin{displaymath} \leq c'' \cdot \big(L^{3/2} + L\cdot L^{1/2}\big)= c''\cdot L^{3/2}, \end{displaymath}
where $c''$ is a constant independent of $\mathfrak{L}$.

The proof of Theorem \ref{1.1} is now complete.

\end{proof}

\section{The case of more general curves}\label{4}

In this section we extend the definition of a joint to a more general setting. More particularly, let $\mathcal{F}$ be the family of all non-empty sets in $\R^3$ with the property that, if $\gamma \in \mathcal{F}$ and $x \in \gamma$, then a basic neighbourhood of $x$ in $\gamma$ is either $\{x\}$ or the finite union of parametrised curves, each homeomorphic to a semi-open line segment with one endpoint the point $x$. In addition, if there exists a parametrisation $f:[0,1)\rightarrow \R^3$ of one of these curves, with $f(0)=x$ and $f'(0)\neq 0$, then the line in $\R^3$ passing through $x$ with direction $f'(0)$ is tangent to $\gamma$ at $x$. If $\Gamma \subset \mathcal{F}$, we denote by $T_x^{\Gamma}$ the set of directions of all tangent lines at $x$ to the sets of $\Gamma$ passing through $x$ (note that $T_x^{\Gamma}$ might be empty and that there might exist many tangent lines to a set of $\Gamma$ at $x$).

Real algebraic curves in $\R^3$, as well as curves in $\R^3$ parametrised by real polynomials, belong to the family $\mathcal{F}$.

\textbf{Definition.} \textit{Let $\Gamma$ be a collection of sets in $\mathcal{F}$.  Then a point $x$ in $\R^3$ is a joint for the collection $\Gamma$ if}

(i)\textit{ $x$ belongs to at least one of the sets in $\Gamma$, and}\newline
(ii)\textit{ there exist at least 3 vectors in $T_x^{\Gamma}$ spanning $\R^3$.}

\textit{The multiplicity $N(x)$ of the joint $x$ is defined as the number of triples of lines in $\R^3$ passing through $x$, whose directions are linearly independent vectors in $T_x^{\Gamma}$.}

We will show here that, under certain assumptions on the characteristics of the sets in a finite collection $\Gamma \subset \mathcal{F}$, the statement of Theorem \ref{1.1} still holds, i.e. \begin{displaymath} \sum_{x \in J}N(x)^{1/2} \leq c \cdot |\Gamma|^{3/2}, \end{displaymath} where $J$ is the set of joints formed by $\Gamma$. To that end, we will need to recall and further analyse some facts from algebraic geometry.

\subsection{Analysing the geometric background}

If $\mathbb{K}$ is a field, then any set of the form \begin{displaymath} \{x \in \mathbb{K}^n: p_i(x)=0, \; \forall \; i=1,...,k\}, \end{displaymath} where $k \in \N$ and $p_i \in \mathbb{K}[x_1,...,x_n]$ for all $i =1,...,k$, is called an \textit{algebraic set} or an \textit{affine variety} or simply a \textit{variety} in $\mathbb{K}^n$, and is denoted by $V(p_1,...,p_k)$. A variety $V$ in $\mathbb{K}^n$ is \textit{irreducible} if it cannot be expressed as the union of two non-empty varieties in $\mathbb{K}^n$ which are strict subsets of $V$.

Now, if $V$ is a variety in $\mathbb{K}^n$, the set \begin{displaymath} I(V):= \{ p \in \mathbb{K}[x_1,...,x_n]: p(x)=0, \; \forall \; x \in V\} \end{displaymath}
is an ideal in $\mathbb{K}[x_1,...,x_n]$. If, in particular, $V$ is irreducible, then $I(V)$ is a prime ideal of $\mathbb{K}[x_1,...,x_n]$, and the transcendence degree of the ring $\mathbb{K}[x_1,...,x_n]/I(V)$ over $\mathbb{K}$ is the \textit{dimension} of the irreducible variety $V$. The \textit{dimension of an algebraic set} is the maximal dimension of all the irreducible varieties contained in the set. If an algebraic set has dimension 1 it is called an \textit{algebraic curve}, while if it has dimension $n-1$ it is called an \textit{algebraic hypersurface}.

Now, if $\gamma$ is an algebraic curve in $\C^n$, a generic hyperplane of $\C^n$ intersects the curve in a specific number of points (counted with appropriate multiplicities), which is called the \textit{degree} of the curve.

A consequence of B\'ezout's theorem (see, for example, \cite[Theorem 12.3]{MR732620} or \cite[Chapter 3, \S3]{MR2122859}) is the following.

\begin{theorem}\emph{\textbf{(B\'ezout)}}\label{4.1.2} Let $\gamma$ be an irreducible algebraic curve in $\C^n$ of degree $b$, and $p \in \C[x_1,...,x_n]$. If $\gamma$ is not contained in the zero set of $p$, it intersects the zero set of $p$ at most $b \cdot \deg p$ times. \end{theorem}

Now, if $\mathbb{K}$ is a field, an order $\prec$ on the set of monomials in $\mathbb{K}[x_1,...,x_n]$ is called a \textit{term order} if it is a total order on the monomials of $\mathbb{K}[x_1,...,x_n]$, such that it is multiplicative (i.e. it is preserved by multiplication by the same monomial) and the constant monomial is the $\prec$-smallest monomial. Then, if $I$ is an ideal in $\mathbb{K}[x_1,...,x_n]$, we define $in_{\prec}(I)$ as the ideal of $\mathbb{K}[x_1,...,x_n]$ generated by the $\prec$-initial terms, i.e. the $\prec$-largest monomial terms, of all the polynomials in $I$. 

Let $V$ be a variety in $\mathbb{K}[x_1,...,x_n]$ and $\prec$ a term order on the set of monomials in $\mathbb{K}[x_1,...,x_n]$. Also, let $S$ be a maximal subset of the set of variables $\{x_1,...,x_n\}$, with the property that no monomial in the variables in $S$ belongs to $in_{\prec}(I(V))$. Then, it holds that the dimension of $V$ is the cardinality of $S$ (see \cite{Sturmfels_2005}). From this fact, we deduce the following.

\begin{lemma}\label{extr}An irreducible real algebraic curve $\gamma$ in $\R^n$ is contained in an irreducible complex algebraic curve in $\C^n$. \end{lemma}

\begin{proof}  We clarify that, by saying that a real algebraic curve $\gamma_1$ in $\R^n$ is contained in a complex algebraic curve $\gamma_2$ in $\C^n$, we mean that, if $x \in \gamma_1$, then the point $x$, seen as an element of $\C^n$, belongs to $\gamma_2$ as well. 

Let $\gamma$ be an irreducible real algebraic curve in $\R^n$, and $\prec$ a term order on the set of monomials in the variables $x_1$, ..., $x_n$. From the discussion above, for every $i\neq j$, $i,j \in \{1,...,n\}$, there exists a monomial in the variables $x_i$ and $x_j$ in the ideal $in_{\prec}(I(\gamma))$. 

Now, the ideal $I(\gamma)$ is finitely generated, like any ideal of $\R[x_1,...,x_n]$. Let $\{p_1,...,p_k\}$ be a finite set of generators of $I(\gamma)$, and let $I':=(p_1,...,p_k)$ be the ideal in $\C[x_1,...,x_n]$ generated by the polynomials $p_1$, ..., $p_k$, this time seen as elements of $\C[x_1,...,x_n]$. We consider the complex variety $V'=V(p_1,...,p_k)$ and the ideal $I(V')$ in $\C[x_1,...,x_n]$. Since the polynomials in $I(\gamma)$, seen as elements of $\C[x_1,...,x_n]$, are elements of $I(V')$, it holds that for every $i\neq j$, $i,j \in \{1,...,n\}$, there exists a monomial in the variables $x_i$ and $x_j$ in the ideal $in_{\prec}(I(V'))$. Therefore, the variety $V'$ has dimension 1 (it cannot have dimension 0, as it is not a finite set of points). 

Therefore, $\gamma$ is contained in a complex algebraic curve. It is finally easy to see by B\'ezout's theorem that $\gamma$ is contained in an irreducible component of that curve.

\end{proof}

Now, by B\'ezout's theorem, we can deduce the following.

\begin{corollary}\label{4.1.4} Let $\gamma _1$, $\gamma_2$ be two distinct irreducible complex algebraic curves in $\C^n$. Then, they have at most $\deg \gamma_1\cdot \deg \gamma_2$ common points. \end{corollary}

\begin{proof} Since $\gamma_2$ is an algebraic curve in $\C^n$ and $\C$ is an algebraically closed field, it follows that $\gamma_2$ is the intersection of the zero sets of  $\lesssim_{n, \deg \gamma_2}1$ irreducible polynomials in $\C[x_1,...,x_n]$, each of which has degree at most $\deg \gamma_2$ (see \cite[Theorem A.3]{MR2827010}). The zero set of at least one of these polynomials does not contain $\gamma_1$, so, by Theorem \ref{4.1.2}, $\gamma_1$ intersects it at most $\deg \gamma_1\cdot \deg \gamma_2$ times. Therefore, $\gamma_1$ intersects $\gamma_2$, which is contained in the zero set of the above-mentioned polynomial, at most $\deg \gamma_1\cdot \deg \gamma_2$ times.

\end{proof}

Corollary \ref{4.1.4} easily implies the following.

\begin{lemma}\label{4.1.1} An irreducible real algebraic curve $\gamma$ in $\R^n$ is contained in a unique irreducible complex algebraic curve in $\C^n$.
\end{lemma}

\begin{proof} Let $\gamma$ be a real algebraic curve in $\R^n$. By Lemma \ref{extr}, $\gamma$ is contained in an irreducible complex algebraic curve in $\C^n$. Suppose that there exist two irreducible complex algebraic curves $\gamma_1$ and $\gamma_2$ in $\C^n$ containing $\gamma$. Then, $\gamma_1$ and $\gamma_2$ intersect at infinitely many points, and thus, by Corollary \ref{4.1.4}, they coincide.

\end{proof}

Note that, by the above, the smallest complex algebraic curve containing a real algebraic curve is the union of the irreducible complex algebraic curves, each of which contains an irreducible component of the real algebraic curve.

In particular, the following holds.

\begin{lemma} \label{newstuff} Any real algebraic curve in $\R^n$ is the intersection of $\R^n$ with the smallest complex algebraic curve containing it. 
\end{lemma}

\begin{proof} Let $\gamma$ be a real algebraic curve in $\R^n$ and $\gamma _{\C}$ the smallest complex algebraic curve containing it. We will show that $\gamma = \R^n \cap \gamma _{\C}$. 

Let $x \in \R^n$, such that $x \notin \gamma$; then, $x \notin \gamma_{\C}$. Indeed, $\gamma$ is the intersection of the zero sets, in $\R^n$, of some polynomials $p_1$, ..., $p_k \in \R[x_1,...,x_n]$. Since $x \notin \gamma$, it follows that $x$ does not belong to the zero set of $p_i$ in $\R^n$, for some $i \in \{1,...,k\}$. However, $x\in \R^n$, so it does not belong to the zero set of $p_i$ in $\C^n$, either. 

Now, the zero set of $p_i$ in $\C^n$ is a complex algebraic set containing $\gamma$, and therefore its intersection with $\gamma_{\C}$ is a complex algebraic set containing $\gamma$; in fact, it is a complex algebraic curve, since it contains the infinite set $\gamma$ and lies inside the complex algebraic curve $\gamma _{\C}$. Therefore, the intersection of the zero set of $p_i$ in $\C^n$ and $\gamma_{\C}$ is equal to $\gamma _{\C}$, as otherwise it would be a complex algebraic curve, smaller that $\gamma_{\C}$, containing $\gamma$. This means that $\gamma_{\C}$ is contained in the zero set of $p_i$ in $\C^n$, and since $x$ does not belong to the zero set of $p_i$ in $\C^n$, it does not belong to $\gamma _{\C}$ either.

Therefore, $\gamma = \R^n \cap \gamma _{\C}$.

\end{proof}

Now, even though a generic hyperplane of $\C^n$ intersects a complex algebraic curve in $\C^n$ in a fixed number of points, this is not true in general for real algebraic curves. However, by Lemma \ref{4.1.1}, we can define the \textit{degree of an irreducible real algebraic curve} in $\R^n$ as the degree of the (unique) irreducible complex algebraic curve in $\C^n$ containing it. Furthermore, we can define the \textit{degree of a real algebraic curve} in $\R^n$ as the degree of the smallest complex algebraic curve in $\C^n$ containing it. With this definition, and due to Lemma \ref{newstuff}, the degree of a real algebraic curve in $\R^n$ is equal to the sum of the degrees of its irreducible components (Lemma \ref{newstuff} ensures that distinct irreducible components of a real algebraic curve in $\R^n$ are contained in distinct irreducible complex algebraic curves in $\C^n$).

Therefore, if, by saying that a real algebraic curve $\gamma$ in $\R^n$ \textit{crosses itself at the point }$x_0 \in \gamma$, we mean that any neighbourhood of $x_0$ in $\gamma$ is homeomorphic to at least two intersecting lines, it follows that a real algebraic curve in $\R^n$ crosses itself at a point at most as many times as its degree.

An immediate consequence of the discussion above is the following.

\begin{corollary}\label{4.1.3} Let $\gamma$ be an irreducible real algebraic curve in $\R^n$ of degree $b$, and $p \in \R[x_1,...,x_n]$. If $\gamma$ is not contained in the zero set of $p$, it intersects the zero set of $p$ at most $b \cdot \deg p$ times. \end{corollary}

We now discuss projections of real algebraic curves. This leads us to the study of semi-algebraic sets.

More particularly, a \textit{basic real semi-algebraic set} in $\R^n$ is any set of the form \begin{displaymath} \{x\in \R^n: P(x)=0 \text{ and }  Q(x)>0, \; \forall \; Q\in \mathcal{Q} \}, \end{displaymath} where $P\in \R[x_1,...,x_n]$ and $\mathcal{Q}$ is a finite family of polynomials in $\R[x_1,...,x_n]$. A \textit{real semi-algebraic set} in $\R^n$ is defined as a finite union of basic real semi-algebraic sets. Note that a real algebraic set in $\R^n$ is, in fact, a basic real semi-algebraic set in $\R^n$, since it can be expressed as the zero set of a single real $n$-variate polynomial (a real algebraic set in $\R^n$ is the intersection of the zero sets, in $\R^n$, of some polynomials $p_1$, ..., $p_k \in \R[x_1,...,x_n]$, which is equal to the zero set, in $\R^n$, of the polynomial $p_1^2+...+p_k^2 \in \R[x_1,...,x_n]$). 

What holds is the following (see \cite[Chapter 2, \S3]{MR2248869} for a proof).

\begin{theorem}\label{4.1.5} The projection of a real algebraic set of $\R^n$ on any hyperplane of $\R^n$ is a real semi-algebraic set. \end{theorem}

We further notice that any set of the form $ \{x\in \R^{n}: Q(x)>0, \; \forall \; Q \in\mathcal{Q}\}$, where $\mathcal{Q}$ is a finite subset of $\R[x_1,...,x_n]$, is open in $\R^{n}$ (with the usual topology). Therefore, a basic real semi-algebraic set in $\R^n$ that is not open in $\R^n$ is of the form $ \{x\in \R^{n}: P(x)=0 \text{ and }  Q(x)>0, \; \forall \; Q\in \mathcal{Q} \}$, where $\mathcal{Q}$ is a finite subset of $\R[x_1,...,x_{n}]$ and $P\in \R[x_1,...,x_n]$ is a non-zero polynomial. Thus, each basic real semi-algebraic set in $\R^n$ that is not open in $\R^n$ (with the usual topology) is contained in a real algebraic set of dimension at most $n-1$.

Now, if $\gamma$ is a real algebraic curve in $\R^3$, its projection on a generic plane $H \simeq \R^2$ is a finite union of basic real semi-algebraic sets which are not open in $H$, so each of them is contained in some real algebraic set of dimension at most 1. However, the projection of a curve in $\R^3$ on a generic plane is not a finite set of points. Therefore, at least one of these basic real semi-algebraic sets is an infinite set of points, contained in some real algebraic curve in $H$. From this fact, as well as a closer study of the algorithm that constitutes the proof of Theorem \ref{4.1.5} as described in \cite[Chapter 2, \S3]{MR2248869}, we can finally see that the projection of $\gamma$ on a generic plane $H$ is the union of at most $B_{\deg \gamma}$ basic real semi-algebraic sets, each of which either consists of at most $B'_{\deg \gamma}$ points or is contained in a real algebraic curve in $H$ of degree at most $B'_{\deg \gamma}$, where $B_{\deg \gamma}$, $B'_{\deg \gamma}$ are integers depending only on the degree $\deg \gamma$ of $\gamma$. Therefore, the following is true.

\begin{lemma}\label{4.1.6} Let $\gamma$ be a real algebraic curve in $\R^3$. There exists an integer $C_{\deg \gamma} \geq \deg \gamma$, depending only on the degree $\deg \gamma$ of $\gamma$, such that the projection of $\gamma$ on a generic plane is contained in a planar real algebraic curve of degree at most $C_{\deg \gamma}$. \end{lemma}

Note that this means that the \textit{Zariski closure} of the projection of a real algebraic curve $\gamma$ of $\R^3$ on a generic plane, i.e. the smallest variety containing that projection, is, in fact, a planar real algebraic curve.

Our aim now is to find an upper bound on the number of times a planar real algebraic curve $\gamma$ crosses itself, and eventually establish an upper bound on the number of times a real algebraic curve in $\R^3$ crosses itself. To that end, we proceed to show that a planar real algebraic curve $\gamma$ is the zero set of a single, square-free bivariate real polynomial, of degree $\lesssim \deg \gamma$.

\begin{lemma} \label{newstuff2} Let $\gamma$ be an irreducible planar complex algebraic curve. Then, $\gamma$ is the zero set, in $\C^2$, of a single, irreducible polynomial $p \in \C[x,y]$, of degree $\leq \deg \gamma$.
\end{lemma}

\begin{proof} Let $\gamma$ be an irreducible planar complex algebraic curve. Then, $\gamma$ is the intersection of the zero sets, in $\C^2$, of some polynomials $p_1$, ..., $p_k \in \C[x,y]$, for $k\lesssim_{\deg \gamma} 1$, of degrees $\leq \deg \gamma$ (see \cite[Theorem A.3]{MR2827010}). 

Now, for all $i=1,...,k$, the zero set of $p_i$ in $\C^2$ contains $\gamma$, and is thus an algebraic set of dimension at least 1; in fact, equal to 1, as otherwise the zero set of $p_i$ would be the whole of $\C^2$ and $p_i$ would be the zero polynomial. Therefore, the zero set of $p_i$ in $\C^2$ is a planar complex algebraic curve containing $\gamma$, for all $i=1,...,k$. Consequently, $\gamma$ is contained, in particular, in the planar complex algebraic curve that is the zero set of $p_1$ in $\C^2$, and, since $\gamma$ is irreducible, it is equal to one of the irreducible components of the zero set of $p_1$, which is the zero set of an irreducible factor of $p_1$. 

Therefore, $\gamma$ is the zero set, in $\C^2$, of a single, irreducible polynomial $p\in \C[x,y]$, of degree $\leq \deg \gamma$.

\end{proof}

We can therefore easily deduce the following.

\begin{corollary} \label{newstuff3} Let $\gamma$ be an irreducible planar real algebraic curve. Then, $\gamma$ is the zero set, in $\R^2$, of a single, irreducible polynomial $p \in \R[x,y]$, of degree $\leq 2\deg \gamma$.
\end{corollary}

\begin{proof} Let $\gamma_{\C}$ be the (unique) irreducible planar complex algebraic curve containing $\gamma$.

Now, by Lemma \ref{newstuff2}, $\gamma_{\C}$ is the zero set, in $\C^2$, of a single, irreducible polynomial $p \in \C[x,y]$, of degree $\leq \deg \gamma_{\C}\;(=\deg \gamma)$. Thus, by Lemma \ref{newstuff}, $\gamma$ is the zero set of $p$ in $\R^2$, and, since the polynomials $p$ and $\bar{p}$ have the same zero set in $\R^2$, $\gamma$ is the zero set, in $\R^2$, of the polynomial $p\bar{p} \in \R[x,y]$, which is irreducible in $\R[x,y]$, since $p$ is irreducible in $\C[x,y]$. 

Therefore, the statement of the Lemma is proved.

\end{proof}

An immediate consequence of Corollary \ref{newstuff3} is the following.

\begin{corollary}\label{newstuff4} Let $\gamma$ be a planar real algebraic curve. Then, $\gamma$ is the zero set, in $\R^2$, of a single, square-free polynomial $p \in \R[x,y]$, of degree $\leq 2\deg \gamma$.

\end{corollary}

We can now bound from above the number of times a planar real algebraic curve crosses itself.

\begin{lemma} \label{newstuff5}Let $\gamma$ be a planar real algebraic curve. Then, $\gamma$ crosses itself at most $4 (\deg \gamma)^2$ times.
\end{lemma}

\begin{proof} By Corollary \ref{newstuff4}, $\gamma$ is the zero set, in $\R^2$, of a single, square-free polynomial $p \in \R[x,y]$, of degree $\leq 2\deg \gamma$. Since $p$ is square-free, $p$ and $\nabla{p}$ do not have a common factor, so, by B\'ezout's theorem, $p$ and $\nabla{p}$ have at most $(\deg p)^2 \leq (2\deg \gamma)^2$ common roots. 

Indeed, if $p=p_1\cdots p_k$, where $p_1$, ..., $p_k \in \R[x,y]$ are irreducible polynomials, then each common root of $p$ and $\nabla{p}$ is a common root of an irreducible factor $p_i$ of $p$, for some $i \in \{1,...,k\}$, and a polynomial $g_i \in \Big\{\frac{\partial p}{\partial x}, \frac{\partial p}{\partial y}\Big\}$, which does not have $p_i$ as a factor. Therefore, the number of common roots of $p$ and $\nabla{p}$ is equal to at most $\sum_{i=1,...,k}r_i$, where, for each $i\in \{1,...,k\}$, $r_i$ is the number of common roots of $p_i$ and $g_i$. However, for all $i\in\{1,...,k\}$, the polynomials $p_i$ and $g_i \in \R[x,y]$ do not have a common factor, and thus, by B\'ezout's theorem, $r_i \leq \deg p_i \cdot \deg g_i \leq \deg p_i \cdot d$. Therefore, the number of common roots of $p$ and $\nabla{p}$ is $\leq \sum_{i=1,...,k}\deg p_i \cdot d\leq d^2$.

But if $\gamma$ crosses itself at a point $x$, then $x$ is a common root of $p$ and $\nabla{p}$, because otherwise $\gamma$ would be a manifold locally around $x$. So, $\gamma$ crosses itself at most $4(\deg \gamma)^2$ times.

\end{proof}

Lemma \ref{newstuff5} immediately gives an upper bound on the number of times a real algebraic curve in $\R^3$ crosses itself.

\begin{lemma}\label{4.1.7} Let $\gamma$ be a real algebraic curve in $\R^3$. Then, $\gamma$ crosses itself at most $4(\deg \overline{\pi(\gamma)})^2$ times, where $\overline{\pi(\gamma)}$ the smallest planar real algebraic curve containing the projection $\pi(\gamma)$ of the curve $\gamma$ on a generic plane (i.e. the curve that constitutes the Zariski closure of $\pi(\gamma)$). \end{lemma}

\begin{proof} Obviously, $\gamma$ crosses itself at most as many times as $\overline{\pi(\gamma)}$ crosses itself, thus, by Lemma  \ref{newstuff5}, at most $4(\deg \overline{\pi(\gamma)})^2$ times.

\end{proof}

We are now ready to establish an analogue of the Szemer\'edi-Trotter theorem for real algebraic curves in $\R^3$. Indeed, the following is known.

\begin{theorem}\emph{\textbf{ (Kaplan, Matou\v{s}ek, Sharir, \cite[Theorem 4.1]{KMS})}}\label{4.1.8} Let $b$, $k$, $C$ be positive constants. Also, let $P$ be a finite set of points in $\R^2$ and $\Gamma$ a finite set of planar real algebraic curves, such that

\emph{(i)} every $\gamma \in \Gamma$ has degree at most $b$, and \newline
\emph{(ii)} for every $k$ distinct points in $\R^2$, there exist at most $C$ distinct curves in $\Gamma$ passing through all of them.

Then, \begin{displaymath} I_{P,\Gamma} \lesssim_{b,k,C} |P|^{k/(2k-1)}|\Gamma|^{(2k-2)/(2k-1)}+|P|+|\Gamma|. \end{displaymath} \end{theorem}

Combining Theorem \ref{4.1.8} with Lemmas \ref{4.1.6} and \ref{4.1.7}, we deduce the following fact on point--real algebraic curve incidences in $\R^3$.

\begin{lemma}\label{4.1.9} Let $b$ be a positive constant. Also, let $\Gamma$ be a finite set of real algebraic curves in $\R^3$, each of degree at most $b$, and $P$ a finite set of points in $\R^3$. Then, there exists a natural number $D_b \geq b^2+1$, depending only on $b$, such that 

\emph{(i)} $I'_{P,\Gamma} \lesssim_b |P|^{D_b/(2D_b-1)}|\Gamma|^{(2D_b-2)/(2D_b-1)}+|P|+|\Gamma|$, where $I'_{P,\Gamma}$ denotes the number of all pairs $(p,\gamma)$ such that $p\in P$, $\gamma \in \Gamma$, $p\in \gamma$ and $p$ is not an isolated point of $\gamma$, and

\emph{(ii)} if there exist $S$ points in $\R^3$, such that each lies in at least $k$ curves of $\Gamma$ which do not have the point as an isolated point, where $k \geq 2$, then $S\lesssim_b {|\Gamma}| ^2/k^{(2D_b-1)/(D_b-1)} + |\Gamma|/k$. \end{lemma}

\begin{proof} Let $\pi :\R^3 \rightarrow H$ be the projection map of $\R^3$ on a generic plane $H \simeq \R^2$. By Lemma \ref{4.1.6} we know that, for all $\gamma \in \Gamma$, $\pi(\gamma)$ is contained in a planar real algebraic curve $\overline{\pi(\gamma)}$ of degree at most $C_b$, where $C_b \geq b$ is an integer depending only on b. Thus, if $\pi(\Gamma):=\{\pi(\gamma):\gamma \in \Gamma\}$ and $Irr(\overline{\pi(\Gamma)}):=\{$1-dimensional irreducible components of $\overline{\pi(\gamma)}: \gamma \in \Gamma\}$, we have that
\begin{displaymath} I'_{P,\Gamma} \leq I'_{\pi(P), \pi(\Gamma)} \leq I_{\pi(P), Irr(\overline{\pi(\Gamma)})} + 4{C_b}^2\cdot |Irr(\overline{\pi(\Gamma)})|,\end{displaymath} as, by Lemma \ref{newstuff5}, each curve in $Irr(\overline{\pi(\Gamma)})$ crosses itself at most $4\cdot (\deg \overline{\pi(\gamma)})^2$ $\leq4 C_b^2$ times. In addition, by B\'ezout's theorem, for each $C_b^2 +1$ distinct points of $\R^2$ there exists at most 1 curve in $Irr(\overline{\pi(\Gamma)})$ passing through all of them. The application, therefore, of Theorem \ref{4.1.8} for $k=D_b:=C_b^2+1$, the set $\pi(P)$ of points and the set $Irr(\overline{\pi(\Gamma)})$ of planar real algebraic curves, whose cardinality is obviously $\leq C_b\cdot|\Gamma|$, completes the proof of (i), while (ii) is an immediate corollary of (i).

\end{proof}

For the analysis that follows, we introduce the notion of the resultant of two polynomials, a useful tool for deducing whether two polynomials have a common factor (for details, see \cite[Chapter 3]{MR2122859} or \cite{Guth_Katz_2008}).

More particularly, let $f$, $g \in \C[x]$, of positive degrees $l$ and $m$, respectively, with
\begin{displaymath}f(x)=a_lx^l+a_{l-1}x^{l-1}+...+a_0
\end{displaymath}
and
\begin{displaymath}g(x)=b_mx^m+b_{m-1}x^{m-1}+...+b_0.
\end{displaymath}

We define the resultant $Res(f,g)$ of $f$ and $g$ as the determinant of the $(l+m)  \times  (l+m)$ matrix $(c_{ij})$, where $c_{ij} =a_{j-i}$ if $1 \leq i \leq m$ and $i \leq j \leq i+l$,
$c_{ij}= b_{j-i +m}$ if $m+1 \leq i \leq m+l$ and $i-m \leq j \leq i-m+l$, and $c_{ij}=0$ otherwise.

Note that the columns of the matrix $(c_{ij})$ represent the coefficients of the polynomial
$f$ multiplied by $x^j$, where $j$ runs from 0 to $m-1$, and the
coefficents of the polynomial $g$ multiplied by $x^k$,  where $k$ runs
from 0 to $l-1$. Therefore, the resultant of $f$ and $g$ is 0 if and only if this set of polynomials is linearly independent. This leads to a connection between the existence of a common factor of two polynomials and the value of their resultant.

Indeed, let $f$, $g \in {\C}[x_1,..., x_n]$ be polynomials of positive degree in $x_1$. Viewing $f$ and $g$ as polynomials in $x_1$ with coefficients in $\C[x_2..., x_n]$, we define the resultant of $f$ and $g$ with respect to $x_1$ as the polynomial $Res(f,g;x_1)\in \C[x_2,...,x_n]$.

\begin{theorem} \label{factor} Let $f$, $g \in {\C}[x_1, ... , x_n]$ be polynomials of positive degree in $x_1$. Then, $f$ and $g$ have a common factor of positive degree in $x_1$ if and only if $Res(f,g;x_1)$ is the zero polynomial.
\end{theorem}

Theorem \ref{factor} is \S 3.6 Proposition 1 (ii) in \cite{Cox+Others/1991/Ideals}. In fact, the following is true (see \cite{Cox+Others/1991/Ideals}).

\begin{lemma} \label{GCD} Let $f$, $g \in \C[x_1,...,x_n]$ of positive degree in $x_1$. Then, there exist $A$, $B \in \C[x_2,...,x_n][x_1]$, such that $Res(f,g;x_1)=Af+Bg$.

\end{lemma}

In particular, Lemma \ref{GCD} implies the following.

\begin{lemma} \label{rootsresultant} Let $f$, $g \in \C[x_1,x_2]$ be polynomials of positive degree in $x_1$. If $f$, $g$ both vanish at the point $(r_1,r_2)\in \C^2$, then $Res(f,g;x_1)$ vanishes at $r_2$.

\end{lemma}

On the other hand, by the definition of the resultant of two polynomials, it is easy to see the following (see \cite{Cox+Others/1991/Ideals}).

\begin{lemma} \label{emergency} Let $f$, $g \in \C[x_1,x_2]$ be polynomials of positive degree in $x_1$. Then, $Res(f,g;x_1)$ is a polynomial in $x_2$, of degree at most $\deg f \cdot \deg g$.

\end{lemma}

We are now ready to extend the proof of \cite[Corollary 2.5]{Guth_Katz_2008} to a more general setting, to deduce the following.

\begin{lemma}\label{4.1.12} Suppose that $f$, $g$ are non-constant polynomials in $ \C[x,y,z]$ which do not have a common factor. Then, the number of irreducible complex algebraic curves which are simultaneously contained in the zero set of $f$ and the zero set of $g$ in $\C^3$ is $\leq \deg f \cdot \deg g$. \end{lemma}

\begin{proof} Let $\Gamma$ be the family of irreducible complex algebraic curves in $\C^3$. Suppose that there exist $\deg f \cdot \deg g+1$ curves in $\Gamma$, simultaneously contained in the zero set of $f$ and the zero set of $g$. A generic complex plane intersects a complex algebraic curve in $\C^3$ at least once and finitely many times, while each two curves in $\Gamma$ intersect in finitely many points of $\C^3$. Therefore, we can change the coordinates, so that $f$ and $g$ have positive degree in $x$, and also so that there exists some point $p=(p_1,p_2,p_3)\in \C^3$ and some $\epsilon >0$, such that any plane in the family $\mathcal{A}:=\big\{$planes in $\C^3$, perpendicular to $(0,0,1)$ and passing through a point of the form $p+ \delta \cdot (0,0,1)$, for $\delta \in (-\epsilon, \epsilon)\}$ is transverse to all the $\deg f \cdot \deg g +1$ curves, intersecting them at points with distinct $y$ coordinates. Thus, each such plane contains at least $\deg f \cdot \deg g +1$ points of $\C^3$ with distinct $y$ coordinates, where both $f$ and $g$ vanish.

Therefore, if $\Pi \in \mathcal{A}$, then $f_{|\Pi}$, $g_{|\Pi}$ are two polynomials in $\C[x,y]$, vanishing at $\geq \deg f \cdot \deg g +1 \geq \deg f_{|\Pi} \cdot \deg g_{|\Pi}+1$ points of $\C^2$ with distinct $y$ coordinates.

At the same time, there are at most $\deg f$, i.e. finitely many, planes $\Pi$ in $\mathcal{A}$, such that $f_{|\Pi}$ does not have positive degree in $x$, and at most $\deg g$, i.e. finitely many, planes $\Pi$ in $\mathcal{A}$, such that $g_{|\Pi}$ does not have positive degree in $x$. Indeed, suppose that there are more than $\deg f$ planes in $\mathcal{A}$, such that $f_{|\Pi}$ does not have positive degree in $x$. Let $h$ be the coefficient of a positive power of $x$ in the expression of the polynomial $f$ as a polynomial of $x$. We view $h$ as a complex polynomial in $x$, $y$, $z$, of non-positive degree in $x$. By our assumption, there are more than $\deg f\geq \deg h$ planes in $\mathcal{A}$ on which $h$ vanishes, therefore $h$ vanishes on $\C^3$ (since a generic line in $\C^3$ intersects all those planes, and thus lies in the zero set of $h$). Hence, $h$ is the zero polynomial; and since $h$ was the coefficient of an arbitrary positive power of $x$ in the expression of $f$ as a polynomial of $x$, it follows that $f$ does not have positive degree in $x$, which is a contradiction. We similarly get a contradiction if we assume that there exist more than $\deg g$ planes in $\mathcal{A}$, such that $g_{|\Pi}$ does not have positive degree in $x$.

Thus, there exists an open interval $I \subset (-\epsilon, \epsilon)$, such that, if $\Pi$ is a plane in the family $\mathcal{A}':=\big\{$planes in $\C^3$, perpendicular to $(0,0,1)$ and passing through a point of the form $p+\delta \cdot (0,0,1)$, for $\delta \in I\}$, then $f_{|\Pi}$, $g_{|\Pi}$ are two polynomials in $\C[x,y]$, of positive degree in $x$, vanishing at $\geq \deg f_{|\Pi} \cdot \deg g_{|\Pi}+1$ points of $\C^2$ with distinct $y$ coordinates. Thus, by Lemma \ref{emergency}, $Res(f_{|\Pi},g_{|\Pi};x)\equiv 0$, for all $\Pi \in \mathcal{A}'$. However, $Res(f_{|\Pi}, g_{|\Pi};x)\equiv Res(f,g;x)_{|\Pi}$, for all $\Pi \in \mathcal{A}'$.

As a result, we have that, for all $\Pi \in \mathcal{A}'$, $Res(f,g;x)_{|\Pi} \equiv 0$; this means that the polynomial $Res(f,g;x) \in \C[y,z]$ vanishes for all $(y,z) \in \C^2$, such that $y \in \C$ and $z \in J$, for some subset $J$ of $\C$ of the form $\{z \in \C: z=p_3+a$, for $a \in (a_1, a_2)\}$, where $a_1$, $a_2 \in \R$. In other words, the polynomial $Res(f,g;x) \in \C[y,z]$ vanishes on a rectangle of $\C^2$. Therefore, it vanishes identically. And $Res(f,g;x) \equiv 0$ means that $f$ and $g$ have a common factor (since they both have positive degree when viewed as polynomials in $x$). We are thus led to a contradiction, which means that there exist $\leq  \deg f \cdot \deg g$ curves of $\Gamma$ simultaneously contained in the zero set of $f$ and the zero set of $g$.

\end{proof}

\begin{corollary}\label{4.1.13} Let $f$ and $g$ be non-constant polynomials in $\R[x,y,z]$. Suppose that $f$ and $g$ do not have a common factor. Then, the number of irreducible real algebraic curves which are simultaneously contained in the zero set of $f$ and the zero set of $g$ in $\R^3$ is $\leq \deg f \cdot \deg g$. \end{corollary}

\begin{proof} We see $f$ and $g$ as polynomials in $\C[x,y,z]$, and viewed as such we denote them by $f_{\C}$,  $g_{\C}$, respectively. Also, let $\Gamma$ be the family of irreducible real algebraic curves in $\R^3$. For all $\gamma \in \Gamma$, we denote by $\gamma_{\C}$ the (unique) irreducible complex algebraic curve containing $\gamma$. 

Since the polynomials $f$, $g\in \R[x,y,z]$ do not have a common factor in $\R[x,y,z]$, the polynomials $f_{\C}$, $g_{\C}\in \C[x,y,z]$ do not have a common factor in $\C[x,y,z]$. Indeed, if $h \in \C[x,y,z]$ was a common factor of $f_{\C}$, $g_{\C}$, which are polynomials with real coefficients, then $\bar{h} \in \C[x,y,z]$ would also be a common factor of $f_{\C}$, $g_{\C}$, therefore $h\bar{h} \in \R[x,y,z]$ would be a common factor of the polynomials $f$, $g \in \R[x,y,z]$.

Now, suppose that a curve $\gamma \in \Gamma$ lies in both the zero set of $f$ and the zero set of $g$. We know that, as it contains $\gamma$, the irreducible complex algebraic curve $\gamma _{\C}$ intersects the zero set of $f_{\C}$ (which contains the zero set of $f$) infinitely many times; thus, by B\'ezout's theorem, it is contained in the zero set of $f_{\C}$. Similarly, $\gamma_{\C}$ is contained in the zero set of $g_{\C}$.

Moreover, if $\gamma^{(1)}$, $\gamma^{(2)} \in\Gamma$ are such that $\gamma^{(1)}\not\equiv \gamma^{(2)}$, then $\gamma^{(1)}_{\C}\not\equiv \gamma^{(2)}_{\C}$. The reason for this is that $\gamma^{(1)}$ is the intersection of $\gamma^{(1)}_{\C}$ with $\R^3$, while $\gamma^{(2)}$ is the intersection of $\gamma^{(2)}_{\C}$ with $\R^3$.

Thus, if $>\deg f \cdot \deg g$ curves of $\Gamma$ lie simultaneously in the zero set of $f$ and the zero set of $g$, then there exist $> \deg f \cdot \deg g$ irreducible complex algebraic curves in $\C^3$, lying in both the zero set of $f_{\C}$ and the zero set of $g_{\C}$, where $f_{\C}$, $g_{\C} \in \C[x,y,z]$ do not have a common factor in $\C[x,y,z]$. By Lemma \ref{4.1.12} though, this is a contradiction. Therefore, the number of curves in $\Gamma$, i.e. of irreducible real algebraic curves in $\R^3$, which are simultaneously contained in the zero set of $f$ and the zero set of $g$, is $\leq \deg f \cdot \deg g$.
\newline\end{proof}

\textbf{Definition.} \textit{Let $Z$ be the zero set of a polynomial $p \in \R[x,y,z]$. A curve $\gamma$ in $\R^3$ is a }critical curve\textit{ of $Z$ if each point of $\gamma$ is a critical point of $Z$.}

We are now able to deduce the following.

\begin{corollary}\label{4.1.14} The zero set of a polynomial $p \in \R[x,y,z]$ contains at most $(\deg p)^2$ critical irreducible real algebraic curves of $\R^3$. \end{corollary}

\begin{proof} The polynomials $p_{sf}$ and $\nabla{p_{sf}}$, the intersection of the zero sets of which is the set of critical points of the zero set of $p$, do not have a common factor, as $p_{sf}$ is square-free. Therefore, the result follows by Corollary \ref{4.1.13}.

\end{proof}

On a different subject, it is known (see \cite[Chapter 5]{MR2248869}) that each real semi-algebraic set is the finite, disjoint union of path-connected components. We observe the following.

\begin{lemma}\label{4.1.15} A real algebraic curve in $\R^n$ is the finite, disjoint union of $\lesssim_{b,n} 1$ path-connected components. \end{lemma}

\begin{proof} This is obvious by a closer study of the algorithm in \cite[Chapter 5]{MR2248869} that constitutes the proof of the fact that every real semi-algebraic set is the finite, disjoint union of path-connected components.

\end{proof}

Finally, we are interested in curves in $\R^3$ parametrised by $t \rightarrow \big(p_1(t)$, $p_2(t)$, $p_3(t)\big)$ for $t \in \R$, where $p_i \in \R[t]$ for $i=1,2,3$. Note that, although curves in $\C^3$ with a polynomial parametrisation are, in fact, complex algebraic curves of degree equal to the maximal degree of the polynomials realising the parametrisation (see \cite[Chapter 3, \S3]{Cox+Others/1991/Ideals}), curves in $\R^3$ with a polynomial parametrisation are not, in general, real algebraic curves, which is why we treat their case separately.

More particularly, if a curve $\gamma$ in $\R^3$ is parametrised by $t \rightarrow \big(p_1(t)$, $p_2(t)$, $p_3(t)\big)$ for $t \in \R$, where the $p_i \in \R[t]$, for $i=1,2,3$, are polynomials not simultaneously constant, then the complex algebraic curve $\gamma_{\C}$ parametrised by the same polynomials viewed as elements of $\C[t]$ is irreducible (it is easy to see that if it contains a complex algebraic curve, then the two curves are identical). Therefore, by B\'ezout's theorem, $\gamma_{\C}$ is the unique complex algebraic curve containing $\gamma$.

Taking advantage of this fact, we will show here that each curve in $\R^3$ with a polynomial parametrisation is contained in a real algebraic curve in $\R^3$.

To that end, we first show the following.

\begin{lemma} \label{maybelast}The intersection of a complex algebraic curve in $\C^n$ with $\R^n$ is a real algebraic set, of dimension at most 1.
\end{lemma}

\begin{proof} Let $\gamma_{\C}$ be a complex algebraic curve in $\C^n$, and $\gamma$ the intersection of $\gamma_{\C}$ with $\R^n$. We show that $\gamma$ is a real algebraic set, of dimension at most 1.

Indeed, since $\gamma_{\C}$ is a complex algebraic set in $\C^n$, there exist polynomials $p_1$, ..., $p_k \in \C[x_1,...,x_n]$, such that $\gamma_{\C}$ is the intersection of the zero sets of $p_1$, ..., $p_k$ in $\C^n$. Now, for $i=1,...,k$, the intersection of the zero set of the polynomial $p_i$ in $\C^n$ with $\R^n$ is equal to the zero set of $p_i$ in $\R^n$, which is the same as the zero set of the polynomial $p_i\bar{p_i} \in \R[x_1,...,x_n]$ in $\R^n$. Therefore, $\gamma$ is the intersection of the zero sets of the polynomials $p_1\bar{p_1}$, ..., $p_k\bar{p_k} \in \R[x_1,...,x_n]$ in $\R^n$, it is thus a real algebraic set.

Moreover, let $\prec$ be a term order on the set of monomials in the variables $x_1$, ..., $x_n$. Since $\gamma_{\C}$ is a complex algebraic curve in $\C^n$, it holds that, for every $i\neq j$, $i,j \in \{1,...,n\}$, there exists a monomial in the variables $x_i$ and $x_j$ in the ideal $in_{\prec}(I(\gamma_{\C}))$. Now, if a polynomial $p \in \C[x_1,...,x_n]$ belongs to $I(\gamma_{\C})$, i.e. vanishes on $\gamma_{\C}$, then the polynomial $p \bar{p} \in \R[x_1,...,x_n]$ vanishes on $\gamma$, and thus belongs to the ideal $I(\gamma)$ of $\R[x_1,...,x_n]$. Therefore, there exists a monomial in the variables $x_i$ and $x_j$ in the ideal $in_{\prec}(I(\gamma))$. Consequently, the algebraic set $\gamma$ has dimension at most 1. 

\end{proof}

\begin{corollary} \label{parametrisedcurves} Let $\gamma$ be a curve in $\R^3$, parametrised by $t \rightarrow \big(p_1(t)$, $p_2(t)$, $p_3(t)\big)$ for $t \in \R$, where the $p_i \in \R[t]$, for $i=1,2,3$, are polynomials not simultaneously constant, of degree at most $b$. Then, $\gamma$ is contained in an irreducible real algebraic curve in $\R^3$, of degree at most $b$.
\end{corollary}

\begin{proof} Let $\gamma_{\C}$ be the curve in $\C^3$, parametrised by $t \rightarrow \big(p_1(t)$, $p_2(t)$, $p_3(t)\big)$ for $t \in \C$. As we have already discussed, $\gamma_{\C}$ is the (unique) irreducible complex algebraic curve containing $\gamma$.

Clearly, $\gamma$ is contained in the intersection of $\gamma_{\C}$ with $\R^3$, which, by Lemma \ref{maybelast}, is a real algebraic set, of dimension at most 1. However, since $\gamma_{\C} \cap \R^3$ contains the parametrised curve $\gamma$, it has, in fact, dimension equal to 1. Therefore, $\gamma$ is contained in the real algebraic curve $\gamma_{\C}\cap\R^3$. In fact, $\gamma_{\C}\cap\R^3$ is an irreducible real algebraic curve. Indeed, if $\gamma ' \subsetneq \gamma_{\C}'\cap\R^3$ was an irreducible real algebraic curve, and $\gamma_{\C}'$ was the (unique) irreducible complex algebraic curve containing it, then $\gamma_{\C}'\cap \R^3 \supsetneq \gamma'=\gamma_{\C}'\cap \R^3$, and thus $\gamma_{\C}\cap \gamma_{\C}'$ $\subsetneq \gamma_{\C}$ would be a complex algebraic curve, which cannot hold, since $\gamma_{\C}$ is irreducible.

Moreover, since $\gamma_{\C}$ is an irreducible algebraic curve in $\C^3$, it is the smallest complex algebraic curve containing the real algebraic curve $\gamma_{\C} \cap \R^3$, and thus the degree of $\gamma_{\C}\cap\R^3$ is equal to $\deg \gamma_{\C}=\max\{\deg p_1, \deg p_2, \deg p_3\}$, and thus equal to at most $b$.

\end{proof}

Note that the Szemer\'edi-Trotter type theorem \ref{4.1.8} gives upper bounds on incidences between points and real algebraic curves in $\R^2$, of uniformly bounded degree. However, it cannot be extended to hold for a family of curves in $\R^2$ parametrised by real univariate polynomials of uniformly bounded degree, without extra hypotheses on the family of the cuvres. Indeed, half-lines in $\R^2$ are such curves, and if the Szemer\'edi-Trotter theorem held for any finite collection of half-lines in $\R^2$, then the Szemer\'edi-Trotter theorem for lines would be scale invariant, since there exist infinitely many distinct half-lines lying on the same line.

Similarly, there does not exist, in general, an upper bound on the number of critical curves parametrised by real univariate polynomials of uniformly bounded degree, contained in an algebraic hypersurface in $\R^3$.

\subsection{The general result}

We are now ready to formulate the following extension of Theorem \ref{1.1}.

\begin{theorem}\label{4.2.1} Let $b$ be a positive constant and $\Gamma$ a finite collection of real algebraic curves in $\R^3$, of degree at most $b$. Let $J$ be the set of joints formed by $\Gamma$. Then,
\begin{displaymath} \sum_{x \in J}N(x)^{1/2} \leq c_b \cdot |\Gamma|^{3/2},\end{displaymath}
where $c_b$ is a constant depending only on $b$. \end{theorem}

The proof of Theorem \ref{4.2.1} is completely analogous to the proof of Theorem \ref{1.1}. Indeed, if $\gamma$ is a real algebraic curve in $\R^3$, of degree at most $b$, and $x\in \gamma$ is not an isolated point of $\gamma$, then $\gamma$ crosses itself at $x$ at most $b$ times, while there exists at least one tangent line to $\gamma$ at $x$; thus, there exist at least 1 and at most $b$ tangent lines to $\gamma$ at $x$. So, if $x$ is a joint of multiplicity $N$ for $\Gamma$, such that at most $k$ curves of $\Gamma$, of which $x$ is not an isolated point, are passing through $x$, then $N \leq (bk)^3$. Therefore, the following lemmas hold, whose statements and proofs are analogous to those of Lemmas \ref{3.1} and \ref{3.2}.

\begin{lemma}\label{4.2.2} Let $x$ be a joint of multiplicity $N$ for a finite collection $\Gamma$ of real algebraic curves in $\R^3$, of degree at most $b$. Suppose that $x$ lies in $\leq 2k$ of the curves in $\Gamma$ of which it is not an isolated point. If, in addition, $x$ is a joint of multiplicity $\leq N/2$ for a subcollection $\Gamma '$ of $\Gamma$, or if it is not a joint at all for the subcollection $\Gamma '$, then there exist $\geq \frac{N}{1000b^3\cdot k^2}$ curves of $\Gamma \setminus \Gamma '$, of which $x$ is not an isolated point, passing through $x$. \end{lemma}

\begin{lemma}\label{4.2.3} Let $x$ be a joint of multiplicity $N$ for a finite collection $\Gamma$ of real algebraic curves in $\R^3$, of degree at most $b$. Suppose that $x$ lies in $\leq 2k$ of the curves in $\Gamma$ of which it is not an isolated point. Then, for every plane containing $x$, there exist $\geq \frac{N}{1000b^3 \cdot k^2}$ curves in $\Gamma$, such that their tangent vectors at $x$ are well-defined and not parallel to the plane. \end{lemma}

Now, for a collection $\Gamma$ of real algebraic curves in $\R^3$, if $J$ is the set of joints formed by $\Gamma$, we define

$J_N:=\{x \in J: N \leq N(x) <2N\}$, for all $N \in \N$, and

$J_N^{k}:=\{x \in J_N$: $x$ intersects at least $k$ and fewer than $2k$ curves of $\Gamma$ of which $x$ is not an isolated point$\}$, for all $N$, $k \in \N$.

Then, Theorem \ref{4.2.1} easily follows from Proposition \ref{4.2.4}, the statement and a sketch of the proof of which we now present.

\begin{proposition}\label{4.2.4} Let $b \in \N$ and $\Gamma$ a finite collection of real algebraic curves in $\R^3$, of degree at most $b$. Then, \begin{displaymath} |J^{k}_{N}| \cdot N^{1/2} \leq c_b \cdot \bigg(\frac{|\Gamma|^{3/2}}{k^{1/(2D_b-2)}}+ \frac{|\Gamma|}{k}\cdot N^{1/2}\bigg), \end{displaymath}
where $D_b$ and $c_b$ are constants depending only on $b$ (and, in particular, $D_b\geq b^2+1$). \end{proposition}

\begin{proof} Each real algebraic curve in $\R^3$ of degree at most $b$ consists of $\leq b \lesssim _b 1$ irreducible components; we may therefore assume that each $\gamma \in \Gamma$ is irreducible. 

Keeping in mind that a real algebraic curve of degree at most $b$ in $\R^3$ cosses itself at a point $x$ at most $b$ times, and therefore the number of tangent lines to $\gamma$ at $x$ is at most $b$, the proof is completely analogous to that of Proposition 1.2. The main differences lie at the beginning and the cellular case, we thus go on to point them out.

By Lemma \ref{4.1.6}, there exists an integer $C_b\geq b$, such that, if $\gamma$ is a real algebraic curve in $\R^3$, of degree at most $b$, then the projection of $\gamma$ on a generic plane is contained in a planar real algebraic curve, of degree $\leq C_b$.

Therefore, by Lemmas \ref{4.1.7} and \ref{4.1.9}, the integer $D_b:=C_b^2+1$ has the following properties.

(i) If $\gamma$ is a real algebraic curve in $\R^3$, of degree at most $b$, then $\gamma$ crosses itself at most $4D_b$ times.

(ii) There exists at most 1 real algebraic curve in $\R^3$, of degree at most $b$, passing through any fixed $D_b$ points in $\R^3$.

(iii) For any finite collection $\Gamma$ of real algebraic curves in $\R^3$, of degree at most $b$, it holds that $|J_N^k|\lesssim_b |\Gamma| ^2/k^{(2D_b-1)/(D_b-1)} + |\Gamma|/k$.

This will be the integer $D_b$ appearing in the statement of the Proposition.

Now, the proof of the Proposition will be achieved by induction on the cardinality of $|\Gamma|$. Indeed, let $M \in \N$. For $c_b$ an explicit constant $\geq D_b$, which depends only on $b$ and will be specified later:

- For any collection $\Gamma$ of irreducible real algebraic curves in $\R^3$, of degree at most $b$, such that $|\Gamma|=1$, it holds that \begin{displaymath} |J^k_{N}|\cdot N^{1/2} \leq c_b \cdot \bigg( \frac{1^{3/2}}{k^{1/(2D_b-2)}} + \frac{1}{k}\cdot N^{1/2}\bigg), \; \forall \; N, \;k \in \N\end{displaymath} (this is obvious, in fact, for any $c_b \geq 4D_b$, as in this case $|J_{N}|=|J_N^1|\leq 4D_b$ for all $N \in \N$, since a real algebraic curve in $\R^3$, of degree at most $b$, crosses itself at most $4D_b$ times). 

- We assume that \begin{displaymath} |J_N^{k}| \cdot N^{1/2} \leq c_b \cdot \bigg(\frac{|\Gamma|^{3/2}}{k^{1/(2D_b-2)}} + \frac{|\Gamma|}{k}\cdot N^{1/2}\bigg), \; \forall \; N,\;k \in \N,\end{displaymath} for any finite collection $\Gamma $ of irreducible real algebraic curves in $\R^3$, of degree at most $b$, such that $|\Gamma| \lneq M.$

- We will now prove that  \begin{equation} |J_N^{k}| \cdot N^{1/2} \leq c_b \cdot \bigg(\frac{|\Gamma|^{3/2}}{k^{1/(2D_b-2)}} + \frac{|\Gamma|}{k}\cdot N^{1/2}\bigg),\; \forall \; N, \;k \in \N, \label{eq:final'}\end{equation}
for any collection $\Gamma$ of irreducible real algebraic curves in $\R^3$, of degree at most $b$, such that $|\Gamma|=M$.

Indeed, let $\Gamma$ be a collection of irreducible real algebraic curves in $\R^3$, of degree at most $b$, such that $|\Gamma|=M$.

Fix $N$ and $k$ in $\N$, and let \begin{displaymath}\mathfrak{G}:=J_{N}^k\end{displaymath} and $$ S:=|J_{N}^k|$$ for this collection $\Gamma$.

Now, we know that $S\cdot N^{1/2} \leq c_{0,b} \cdot (|\Gamma| ^2/k^{(2D_b-1)/(D_b-1)} + |\Gamma|/k)$ for some constant $c_{0,b}$ depending only on $b$. Thus:

If $\frac{S}{2}\leq c_{0,b}\cdot  \frac{|\Gamma|}{k}$, then $S \cdot N^{1/2} \leq 2c_{0,b} \cdot \frac{|\Gamma|}{k}\cdot N^{1/2}$ (where $2c_{0,b}$ is a constant depending only on $b$). 

Otherwise, $\frac{S}{2}< c_{0,b} \cdot|\Gamma| ^2/k^{(2D_b-1)/(D_b-1)}$, so $S < 2c_{0,b} \cdot  |\Gamma|^2 k^{-(2D_b-1)/(D_b-1)}$. 

Therefore, $d:=A_b|\Gamma|^2S^{-1}k^{-(2D_b-1)/(D_b-1)}$ is a quantity $>1$ whenever $A_b\geq 2c_{0,b}$; we thus choose $A_b$ to be large enough for this to hold, and we will specify its value later. Now, applying the Guth-Katz polynomial method for this $d>1$ and the finite set of points $\mathfrak{G}$, we deduce that there exists a non-zero polynomial $p\in \R[x,y,z]$, of degree $\leq d$, whose zero set $Z$:

(i) decomposes $\R^3$ in $\lesssim d^3$ cells, each of which contains $\lesssim Sd^{-3}$ points of $\mathfrak{G}$, and

(ii) contains 6 distinct generic planes, each of which contains a face of a fixed cube $Q$ in $\R^3$, such that the interior of $Q$ contains $\mathfrak{G}$ (and each of the planes is generic in the sense that the plane in $\C^3$ containing it intersects the smallest complex algebraic curve in $\C^3$ containing $\gamma$, for all $\gamma \in \Gamma$);

to achieve this, we first fix a cube $Q$ in $\R^3$, with the property that its interior contains $\mathfrak{G}$ and the planes containing its faces are generic in the above sense. Then, we multiply the polynomials we end up with at each step of the Guth-Katz polynomial method with the same (appropriate) six linear polynomials, the zero set of each of which is a plane containing a different face of the cube, and stop the application of the method when we finally get a polynomial of degree at most $d$, whose zero set decomposes $\R^3$ in $\lesssim d^3$ cells (the set of the cells now consists of the non-empty intersections of the interior of the cube $Q$ with the cells that arise from the application of the Guth-Katz polynomial method, as well as the complement of the cube).

We can assume that the polynomial $p$ is square-free, as eliminating the squares of $p$ does not inflict any change on its zero set.

If there are $\geq 10^{-8}S$ points of $\mathfrak{G}$ in the union of the interiors of the cells, we are in the cellular case. Otherwise, we are in the algebraic case.

\textbf{Cellular case:} There are $\gtrsim S$ points of $\mathfrak{G}$ in the union of the interiors of the cells. However, we also know that there exist $\lesssim d^3$ cells in total, each containing $\lesssim Sd^{-3}$ points of $\mathfrak{G}$. Therefore, there exist $\gtrsim d^3$ cells, with $\gtrsim Sd^{-3}$ points of $\mathfrak{G}$ in the interior of each. We call the cells with this property ``full cells". Now:

$\bullet$ If the interior of some full cell contains $< k^{1/(D_b-1)}$ points of $\mathfrak{G}$, then $Sd^{-3} \lesssim k^{1/(D_b-1)}$, and since $N \lesssim b^3 k^3\lesssim_b k^3$, we have that $S \cdot N^{1/2} \lesssim_b |\Gamma|^{3/2}/k^{1/(2D_b-2)}$.

$\bullet$ If the interior of each full cell contains $\geq k^{1/(D_b-1)}$ points of $\mathfrak{G}$, then we will be led to a contradiction by choosing $A_b$ sufficiently large. Indeed:

Consider a full cell and let $\mathfrak{G}_{cell}$ be the set of points of $\mathfrak{G}$ lying in the interior of the cell, $S_{cell}$ the cardinality of $\mathfrak{G}_{cell}$ and $\Gamma_{cell}:= \{\gamma \in \Gamma:\exists \; x \in \gamma\cap \mathfrak{G}_{cell}$, such that $x$ is not an isolated point of $\gamma\}$.

Let $\mathfrak{G}_{cell}'$ be a subset of $\mathfrak{G}_{cell}$ of cardinality $k^{1/(D_b-1)}$. Since each point of $\mathfrak{G}_{cell}$ has at least $k$ curves of $\Gamma_{cell}$ passing through it, there exist at least $k^{D_b/(D_b-1)}$ incidences between $\Gamma_{cell}$ and $\mathfrak{G}_{cell}'$. On the other hand, the curves in $\Gamma_{cell}$ containing at most $D_b-1$ points of $\mathfrak{G}_{cell}'$ contribute at most $(D_b-1) \cdot|\Gamma_{cell}|$ incidences with $\mathfrak{G}_{cell}'$, while through any fixed point of $\mathfrak{G}'_{cell}$ there exist at most $\binom{|\mathfrak{G}_{cell}'|}{D_b-1}$ curves in $\Gamma_{cell}$, each containing at least $D_b$ points of $\mathfrak{G}_{cell}'$, since there exists at most 1 curve in $\Gamma$ passing through any fixed $D_b$ points in $\R^3$; therefore, there exist at most $\binom{|\mathfrak{G}_{cell}'|}{D_b-1}\cdot |\mathfrak{G}_{cell}'|$ incidences between $\mathfrak{G}_{cell}'$ and the curves in $\Gamma _{cell}$, each of which contains at least $D_b$ points of $\mathfrak{G}_{cell}'$. Thus, 
\begin{displaymath}k^{D_b/(D_b-1)} \leq I_{\mathfrak{G}_{cell}', \Gamma_{cell}} \leq (D_b-1) \cdot |\Gamma_{cell}| +\binom{k^{1/(D_b-1)}}{D_b-1} \cdot k^{1/(D_b-1)} \leq \end{displaymath} \begin{displaymath} \leq (D_b-1) \cdot |\Gamma_{cell}| +\frac{1}{(D_b-1)!}\cdot k^{D_b/(D_b-1)}\text{, so}\end{displaymath} \begin{displaymath}|\Gamma_{cell}| \gtrsim_b k^{D_b/(D_b-1)}, \end{displaymath}and thus \begin{displaymath} |\Gamma_{cell}|^2/k^{(2D_b-1)/(D_b-1)} \gtrsim_b |\Gamma_{cell}| /k.\end{displaymath} Note that this approach differs to the one applied in the case of joints formed by lines.

Now, due to our definition of $D_b$ and the fact that each of the points in $\mathfrak{G}_{cell}$ has at least $k$ curves of $\Gamma_{cell}$ passing through it, of each of which it is not an isolated point, we obtain (since $k \geq 3$, and thus $\geq 2$), 
\begin{displaymath} S_{cell} \lesssim_b |\Gamma_{cell}|^2/k^{(2D_b-1)/(D_b-1)} +|\Gamma_{cell}|/ k.\end{displaymath}
Therefore, $S_{cell} \lesssim_b |\Gamma_{cell}|^2/k^{(2D_b-1)/(D_b-1)}$, so, since we are working in a full cell, $Sd^{-3} \lesssim_b |\Gamma_{cell}|^2/k^{(2D_b-1)/(D_b-1)}$, and rearranging we see that 
\begin{displaymath} |\Gamma_{cell}| \gtrsim_b S^{1/2}d^{-3/2}k^{(2D_b-1)/(2D_b-2)}.\end{displaymath}
Furthermore, let $\Gamma_Z$ be the set of curves of $\Gamma$ which are lying in $Z$. Obviously, $\Gamma_{cell} \subset \Gamma \setminus \Gamma_Z$. Moreover, let $\Gamma_{cell}'$ be the set of curves in $\Gamma_{cell}$ such that, if $\gamma \in \Gamma_{cell}'$, there does not exist any point $x$ in the intersection of $\gamma$ with the boundary of the cell, with the property that the induced topology from $\R^3$ to the intersection of $\gamma$ with the closure of the cell contains some open neighbourhood of $x$. Finally, let $I_{cell}$ denote the number of incidences between the boundary of the cell and the curves in $\Gamma_{cell} \setminus \Gamma'_{cell}$.

Now, each of the curves in $\Gamma_{cell}\setminus \Gamma'_{cell}$ intersects the boundary of the cell at at least one point $x$, with the property that the induced topology from $\R^3$ to the intersection of the curve with the closure of the cell contains an open neighbourhood of $x$; therefore, $I_{cell}\geq |\Gamma_{cell}\setminus \Gamma'_{cell}|$ ($=|\Gamma_{cell}|-|\Gamma'_{cell}|$). Also, the union of the boundaries of all the cells is the zero set $Z$ of $p$, and if $x$ is a point of $Z$ which belongs to a curve in $\Gamma$ intersecting the interior of a cell, such that the induced topology from $\R^3$ to the intersection of the curve with the closure of the cell contains an open neighbourhood of $x$, then there exist at most $2b-1$ other cells whose interior is also intersected by the curve and whose boundary contains $x$, such that the induced topology from $\R^3$ to the intersection of the curve with the closure of each of these cells contains some open neighbourhood of $x$. So, if $I$ is the number of incidences between $Z$ and $\Gamma \setminus \Gamma_Z$, and $\mathcal{C}$ is the set of all the full cells (which, in this case, has cardinality $\sim d^3$), then
\begin{displaymath} I \geq \frac{1}{2b} \cdot \sum_{cell \;\in \;\mathcal{C}} I_{cell}\geq \frac{1}{2b}\cdot \sum_{cell \;\in\; \mathcal{C}}(|\Gamma_{cell}|-|\Gamma_{cell}'|). \end{displaymath}

Now, if $\gamma \in \Gamma_{cell}$ ($\supseteq \Gamma'_{cell}$), we consider the (unique, irreducible) complex algebraic curve $\gamma_{\C}$ in $\C^3$ which contains $\gamma$. In addition, let $p_{\C}$ be the polynomial $p$ viewed as an element of $\C[x,y,z]$, and $Z_{\C}$ the zero set of $p_{\C}$ in $\C^3$. The polynomial $p$ was constructed in such a way that $\gamma_{\C}$ intersects each of 6 complex planes, each of which contains one of the real planes in $Z$ that each contain a different face of the cube $Q$; consequently $\gamma_{\C}$ intersects $Z_{\C}$ at least once. Moreover, if $\gamma^{(1)}$, $\gamma^{(2)}$ are two distinct curves in $\Gamma$, then $\gamma^{(1)}_{\C}$, $\gamma^{(2)}_{\C}$ are two distinct curves in $\Gamma_{\C}$ (since $\gamma^{(1)}=\gamma^{(1)}_{\C}\cap \R^3$, while $\gamma^{(1)}=\gamma^{(1)}_{\C}\cap \R^3$). So, if $\Gamma_{\C}=\{\gamma_{\C}: \gamma \in \Gamma_{cell}\text{, for some cell in }\mathcal{C}\}$ and $I_{\C}$ denotes the number of incidences between $\Gamma_{\C}$ and $Z_{\C}$, it follows that 
$$I_{\C} \geq|\Gamma_{\C}|=|\Gamma|\geq\big|\bigcup_{cell \;\in\; \mathcal{C}}\Gamma_{cell}'\big|,
$$while also
$$I_{\C} \geq I.
$$
Therefore,
$$I_{\C} \geq \frac{1}{2}\cdot (I+I_{\C}) \geq
$$
\begin{displaymath}\geq \frac{1}{2} \cdot \bigg( \frac{1}{2b} \sum_{cell \;\in\; \mathcal{C}}(|\Gamma_{cell}|-|\Gamma_{cell}'|) + \big|\bigcup_{cell \;\in \;\mathcal{C}}\Gamma_{cell}'\big|\bigg) \sim_b \end{displaymath}
\begin{displaymath} \sim_b\sum_{cell \;\in\; \mathcal{C}}(|\Gamma_{cell}|-|\Gamma_{cell}'|) + \big|\bigcup_{cell \;\in \;\mathcal{C}}\Gamma_{cell}'\big|.\end{displaymath} However, each real algebraic curve in $\R^3$, of degree at most $b$, is the disjoint union of $\leq R_b$ path-connected components, for some constant $R_b$ depending only on $b$ (by Lemma \ref{4.1.15}). Therefore, 
$$\big|\bigcup_{cell \;\in \;\mathcal{C}}\Gamma_{cell}'\big| \sim_b\sum_{cell \;\in \;\mathcal{C}}|\Gamma_{cell}'|,
$$
from which it follows that
 \begin{displaymath}I_{\C}\gtrsim_b \sum_{cell \;\in\; \mathcal{C}}(|\Gamma_{cell}|-|\Gamma_{cell}'|)+\sum_{cell \;\in \;\mathcal{C}}|\Gamma_{cell}'|\sim_b \end{displaymath} \begin{displaymath}\sim_b\sum_{cell \;\in \;\mathcal{C}}|\Gamma_{cell}|\gtrsim_b \sum_{cell \;\in \;\mathcal{C}} S^{1/2}d^{-3/2}k^{(2D_b-1)/(2D_b-2)}\sim_b\end{displaymath} \begin{displaymath}\sim_b \big(S^{1/2}d^{-3/2}k^{(2D_b-1)/(2D_b-2)}\big)\cdot d^3 \sim_b S^{1/2}d^{3/2}k^{(2D_b-1)/(2D_b-2)}.\end{displaymath}
On the other hand, however, each $\gamma_{\C}\in \Gamma_{\C}$ is a complex algebraic curve in $\R^3$, of degree at most $b$, which does not lie in $Z_{\C}$, and thus intersects $Z_{\C}$ at most $b \cdot \deg p$ times. So,
$$I_{\C} \lesssim_b |\Gamma_{\C}| \cdot d \sim_b |\Gamma| \cdot d,
$$
and therefore
\begin{displaymath}S^{1/2}d^{3/2}k^{(2D_b-1)/(2D_b-2)} \lesssim_b |\Gamma| \cdot d,\end{displaymath}
which in turn gives $A_b\lesssim_b 1$. In other words, there exists some constant $C_b$, depending only on $b$, such that $A_b \leq C_b$. By fixing $A_b$ to be a constant larger than $C_b$ (and of course $\geq 2c_{0,b}$, so that $d> 1$), we have a contradiction.

Therefore, in the cellular case there exists some constant $c_{1,b}$, depending only on $b$, such that  \begin{displaymath} S \cdot N^{1/2} \leq c_{1,b} \cdot \frac{|\Gamma|^{3/2}}{k^{1/(2D_b-2)}}.\end{displaymath}

\textbf{Algebraic case:} There exist $<10^{-8}S$ points of $\mathfrak{G}$ in the union of the interiors of the cells. 

We denote by $\Gamma'$ the set of curves in $\Gamma$ each of which contains $\geq \frac{1}{100}Sk|\Gamma|^{-1}$ points of $\mathfrak{G}\cap Z$ which are not isolated points of the curve, and we continue by adapting, to this setting, the proof of Proposition \ref{1.2}, using Corollary \ref{4.1.14} and Lemmas \ref{4.2.2} and \ref{4.2.3}. 

\end{proof}

We are now able to count, with multiplicities, joints formed by a finite collection of curves in $\R^3$, parametrised by real univariate polynomials of uniformly bounded degree.

\begin{corollary} \label{parametrisedcurves} Let $b$ be a positive constant and $\Gamma$ a finite collection of curves in $\R^3$, such that each $\gamma \in \Gamma$ is parametrised by $t \rightarrow \big(p^{\gamma}_1(t)$, $p_2^{\gamma}(t)$, $p_3^{\gamma}(t)\big)$ for $t \in \R$, where the $p_i^{\gamma} \in \R[t]$, for $i=1,2,3$, are polynomials not simultaneously constant, of degrees at most $b$. Let $J$ be the set of joints formed by $\Gamma$. Then,
\begin{displaymath} \sum_{x \in J}N(x)^{1/2} \leq c_b \cdot |\Gamma|^{3/2},\end{displaymath}
where $c_b$ is a constant depending only on $b$.
\end{corollary}

\begin{proof} By Corollary \ref{parametrisedcurves}, each $\gamma \in \Gamma$ is contained in a real algebraic curve in $\R^3$, of degree at most $b$. Therefore, the statement of the Corollary immediately follows from Theorem \ref{4.2.1}.

\end{proof}

\end{document}